\documentclass[a4paper]{amsart}
\usepackage[colorinlistoftodos,textsize=footnotesize]{todonotes}
\usepackage[english]{babel}
\usepackage[utf8]{inputenc}
\usepackage{amsmath, amssymb, amsthm}
\usepackage{graphicx}
\usepackage[]{tikz}
\usetikzlibrary{cd, patterns, angles, decorations.pathreplacing, calc, intersections}

\newtheorem{theorem}{Theorem}
\newtheorem{definition}{Definition}
\theoremstyle{plain}
\newtheorem{corollary}{Corollary}

\newtheorem{lemma}{Lemma}
\newtheorem{proposition}{Proposition}

\newtheorem{example}{Example}

\newtheorem{remark}{Remark}
\newtheorem{algorithm}{Algorithm}

\newcommand{\cC}{{\mathcal{C}}}

\newcommand{\cM}{{\mathcal{M}}}

\newcommand{\cS}{{\mathcal{S}}}
\newcommand{\cL}{{\mathcal{L}}}
\newcommand{\cR}{{\mathcal{R}}}

\newcommand{\C}{{\mathbb{C}}}
\newcommand{\Q}{{\mathbb{Q}}}
\newcommand{\Z}{{\mathbb{Z}}}
\newcommand{\R}{{\mathbb{R}}}
\newcommand{\bP}{{\mathbb{P}}}

\newcommand{\fm}{{\mathfrak m}}

\newcommand{\fp}{{\mathfrak p}}

\newcommand{\In}{\operatorname{In}}
\newcommand{\cotg}{\operatorname{cotan}}
\newcommand{\length}{\operatorname{length}}
\newcommand{\tg}{\operatorname{tan}}
\newcommand{\divides}{\bigm|}
\newcommand{\GWT}{\textsc{gwt}}
\newcommand{\WT}{\textsc{wt}}

\title[Combinatorics in resolution of surfaces]{Combinatorics and their evolution in resolution of embedded algebroid surfaces}

\author{Helena Cobo}
\address{Departamento de \'Algebra, Universidad de Sevilla}
\email{helenacobo@gmail.com}
\author{M. J. Soto}
\address{Departamento de \'Algebra, Universidad de Sevilla}
\email{soto@us.es}
\author{Jos\'e M. Tornero}
\address{Departamento de \'Algebra \& IMUS, Universidad de Sevilla}
\email{tornero@us.es}

\date{\today}

\begin{document}

\begin{abstract}
The seminal concept of characteristic polygon of an embedded algebroid surface, first developed  by Hironaka, seems well suited for combinatorially (perhaps even effectively) tracking of a resolution process. However, the way this object evolves through the resolution of singularities is not really well understood, as some references had pointed out.

The aim of this paper is to study how this object changes as the surface gets resolved. In order to get a precise description of the phenomena involved, we need to use different techniques and ideas. Eventually, some effective results regarding the number of blow-ups needed to decrease the multiplicity are obtained as a side product.
\end{abstract}

\maketitle

\section{Introduction}

In this paper we will deal with embedded algebroid surfaces, that is, schemes given by the spectrum of a ring $R = K[[X,Y,Z]]/(F)$, where $K$ is an algebraically closed field of characteristic zero (usually $\C$, but any such field can fill~in) and $F$~is a power series of order~$n>1$. Such an~$F$ will be called an equation of the surface and $n$~will be called the multiplicity of the surface. We are then looking at the local version of the classical algebraic problem.

The resolution of surfaces has a long and fascinating history, as it has traditionally been the guinea pig for testing general resolution procedures. The fact that a surface can be resolved by blowing up was first proved by Walker~\cite{W} although arguably the most influential work in the matter came from O.~Zariski~\cite{Z1,Z2}. In particular, in~\cite{Z2}, he proved the so-called Levi-Zariski Theorem, as a part of his resolution of $3$-dimensional varieties. This result had previously been studied by B.~Levi~\cite{L2}, but a correct formal proof was missing. It essentially states that an embedded surface can be resolved by blowing~up smooth equimultiple centers of maximal dimension until a smooth model is reached. These results and techniques were paralleled by Abhyankar in positive characteristic~\cite{Abh}.

A combinatorial approach to the Levi-Zariski Theorem was pointed out by Hironaka in his seminal work~\cite{H1}, and partially developed in~\cite{Bowdoin}, although the arguments were not very clear in  positive characteristic. Many more years were needed until the works of Cossart and Piltant, \cite{CP1}~and~\cite{CP2}, provided a resolution of singularities of $3$-folds in all characteristics, generalizing \cite{Z2}. Their main argument is first proving local uniformization and then proving that Zariski's method holds in positive characteristic.

As Cossart argues in \cite{CO}, the methods of Jung (a pioneer in the study of surface resolution~\cite{J}), Zariski, Abhyankar and Hironaka are almost the same from the point of view of the characteristic polyhedron of the singularity. In very broad terms, the general strategy is to define a~tuple of non negative integers associated to the surface, the first one typically being the multiplicity, and then to prove that at each blow-up, the resulting tuple is smaller than the original, in the lexicographic ordering.

The relationship between singularities and combinatorics has therefore been both fruitful but also frustrating. 
Though combinatorics still prove to be useful for monitoring the evolution of singularities during the resolution process (see, for instance \cite{HW} for a modern application in positive characteristic), there are still many things we ignore about how to measure the eventual improvement of a singularity through the resolution process by using combinatorial tools.

In order to study the evolution of the combinatorics attached to a surface, a situation which seems to be optimal is that of quadrants. In \cite{HS} problem 6 explicitly asks whether the Newton polyhedron tends to the quadrant situation. This question is posed in the algebraic case, which is not exactly our set up, but it also makes sense in the algebroid situation. As we will show, this happens to be the case, but a relaxed condition, which we call generalized quadrant, is also an interesting intermediate state, which will be needed in order to properly describe the improvement of the combinatorial objects through the resolution process.

In this context, a natural question which might arise, and was in fact the original motivation for our work, is the following: can these combinatorics bound, in some effective sense, the resolution process? For instance, can we know in advance how many times should one expect to blow up the surface before the smooth model is obtained? Or, in more precise terms: is there a bound on the number of times that we need to blow~up the surface, say~following the Levi-Zariski algorithm? 

In the case of curves, it is easy to find such a bound, using the first characteristic exponent as hinted by Hironaka in \cite{H1}. However, the same kind of arguments were not available (due to the involved induction index chosen) for surfaces following neither Hironaka's nor other recent resolutions, such as the taut blow-up achieved by Hauser \cite{HH} or the classical-style developed by Kiyek and Vicente \cite{KV}. In fact, the lexicographic ordering is not well suited at all for these kind of questions, by definition.

As it turns out, the answer for surfaces is not straightforward. In fact, the main combinatorial tool used to date has been the characteristic polygon (Newton, Hironaka or Newton-Hironaka, depending on the source), which can be proved not to be precise enough for our purposes. For one thing, Piedra and the third author~\cite{LZ0} exhibited examples where surfaces with the same polyhedron have very different behavior under blow-ups. This implies that there is not enough information in the characteristic polyhedron (at least, as originally defined) to determine the number of times we need to blow up, and that we must look elsewhere for this information. 

We will give a result in this direction which solves a particular instance of the problem (that is, up to a precise preparation of the surface), and we will explain why it is not reasonable to expect a full answer.

The structure of the paper will be the following:
\begin{itemize}
\item In Section 2 we will review the concept of Newton polyhedron of a singularity, with specific attention to the surface case and the Hironaka polygon attached. We will also define a precise type of equations which will be the ones used in the sequel, called Generalized Weierstrass-Tchirnhausen ({\GWT}) equations.
\item In Section 3 we will review the blow-ups we are going to work with . We will break this study in two pieces: a certain (easy) type of blow-ups and a particular family of changes of variables called transvections.
\item In Section 4 we study the evolution of the polygon under blow-ups, and stress the importance of eliminating compact faces in the Hironaka polygon (that is, taking it into a quadrant) as a combinatorial avatar of the optimal situation, in resolution terms. We will also establish the essential properties of quadrants: its stability during the resolution process and how a generic surface tends in almost all cases to a situation that can be described using  quadrants.
\item In Section 5 we will introduce the fundamental concept of generalized quadrants (which are, essentially, quadrants up to some transvection), and we will prove how a generic surface can be taken, using blow-ups, to another one with {\em better} combinatorial characterization.
\item In Section 6 we will address the problem  of bounding the following resolution process  (essentially, a customized  Levi-Zariski strategy): For an algebroid surface,
\begin{enumerate}
\item if $(Z,X)$ or $(Z,Y)$ are permissible curves, we make a monoidal transformation centered in them,
\item otherwise, we perform a quadratic transformation.
\end{enumerate}
Since the arguments are fundamentally different, we cover first the case of $F$ having a tangent cone which is not a plane, and afterwards the case of a plane tangent cone.
\item Finally, in Section 7 we will comment on how the results from previous sections are optimal, together with displaying some examples.
\end{itemize}

Concerning possible extensions of our results, the positive characteristic case differs from zero characteristic as we will make clear here and there, but a result on this line may still be possible (\cite{HW} may contain hints on how to do that). For higher dimensions, alas,  we know for sure that the Levi-Zariski strategy fails in dimension~3, as proved by~Spivakovsky \cite{Sp2}, so an extension of our results might have an impact in the choice of blow-up centers which must necessarily be less coarse.

Note that all our arguments can be extended almost word for word to both the local analytic case (considering convergent power series instead of formal ones) and the cusp-like algebraic case\footnote{\emph{The most interesting and hardest case}, in the words of Hironaka~\cite{H1}.}. 

\section{Equations and Polyhedra}
\label{S2}

We introduce in this section all the notation and objects that we will use throughout the paper.

\subsection{The Newton Polyhedron and the Hironaka Polygon}

Let us consider the power series ring $K[[X,Y,Z]]$, $K$ an algebraically closed field, and let $\nu(\cdot)$ be the usual order function in this ring. Let $\cS$ be an embedded algebroid surface of multiplicity $n$, and $F(X,Y,Z)$ an equation of $\cS$. In other words, $\cS$ is given by
$$
\cS = \mbox{Spec} \bigl( K[[X,Y,Z]] \bigm/ (F) \bigr), \qquad \text{ with  $n = \nu(F)$.}
$$

\begin{definition}
Let $\overline{F}$ be the initial form of $F$, 
$$
\overline{F} = \sum_{i+j+k = n} a_{ijk} X^iY^jZ^k.
$$
The tangent cone of $\cS$, denoted by $\cC (\cS)$, or simply $\cC$ if no confusion arises, is the projective variety defined by $\overline{F}$ on $\bP^2 (K)$.
\end{definition}

Let now $(\alpha:\beta:\gamma) \in \bP^2(K)$ be a projective point not in $\cC$, that is, such that $\overline{F} (\alpha:\beta:\gamma) \neq 0$, and assume also $\gamma \neq 0$. Such a point always exists, as $K$ must be infinite. Then, the change of variables:
$$
X \longmapsto X + \alpha Z \quad Y \longmapsto Y + \beta Z \quad Z \longmapsto \gamma Z,
$$
takes $F$ into a new power series which is regular in $Z$ of order $n$, that is, it has a monomial $\lambda Z^n$ with $\lambda \in K^*$. Such a series is then associated by the Weierstrass Preparation Theorem with a polynomial (in $K[[X,Y]][Z]$), which is also an equation of $\cS$, as it defines the same ideal than the original $F$.

Therefore, up to a linear change of variables, one can assume $F$ to have the form
$$
F(X,Y,Z) = Z^n + \sum_{k=0}^{n-1} a_k(X,Y) Z^k, 
$$
where $ a_k(X,Y) = \sum_{i,j} a_{ijk} X^iY^j \in K[[X,Y]]$,
with $i+j+k \geq n$ whenever $a_{ijk} \neq 0$. Such an equation will be called a {\em Weierstrass equation} of $\cS$.

\begin{definition}
With the previous notations, let $\fp$ be a prime, non-maximal, ideal on $K[[X,Y,Z]]$ verifying:
\begin{enumerate}
\item[(a)] $F \in \fp^n$.
\item[(b)] There are two power series $G,H \in \fp$ such that
$\mbox{ord}(G) = \mbox{ord} (H) = 1$ and $\fp = (G,H)$.
\end{enumerate}
Such a prime ideal will be called a permissible curve of $\cS$.
\end{definition}

\begin{figure}[htbp]
    \includegraphics[totalheight=6cm]{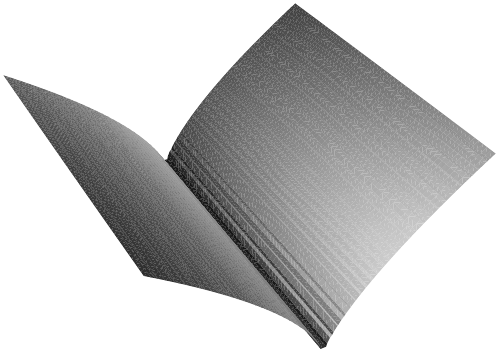}
    \caption{The surface $Z^{2}-X^{3}$ contains the permissible curve $(Z,X)$.}
\end{figure}

This notion of permissible curve agrees with the one derived from normal flatness in the work of Hironaka \cite{H2}. For our surface case, permissible is equivalent to smooth (condition (b)) and equimultiple (condition (a)). Hence, permissible curves are the first option to blow~up when following the Levi-Zariski strategy. These blow-ups are called \emph{monoidal transformations} in classical terminology. In the absence of permissible curves, we are bound to blow up the origin, performing \emph{quadratic transformations}.

Throughout this paper we will follow a resolution process which is very similar to the Levi-Zariski strategy:
\begin{algorithm}
\label{alg}Let $\cS$ be an algebroid surface, given by a Weiestrass equation
$F\in K[[X,Y]][Z]$,
\begin{enumerate}
\item if $(Z,X)$ or $(Z,Y)$ are permissible curves, we make a monoidal transformation centered in them,
\item otherwise, we perform a quadratic transformation.
\end{enumerate}
\end{algorithm}

Note that the algorithm will ignore other, more generic, equimultiple curves, which will be solved “along the way”. This results, unfortunately, in a less optimized algorithm.

\medskip

One of Hironaka's great insights in his work in singularities is that we can attach a combinatorial object to $\cS$. We will quickly review how this is done for surfaces, but see~\cite{Bowdoin, H2} for the details.

\begin{definition}
We will call
$$
N(F) = \left\{ (i,j,k) \in \Z_{\geq 0}^3 \; | \; a_{ijk} \neq 0 \right\} \cup \left\{ (0,0,n) \right\}
$$
the cloud of points of $F$.

The set 
$$
\Gamma (F) = \operatorname{CH} \left[ \bigcup_{(i,j,k) \in N(F)} \Big( (i,j,k) + \Z_{\geq 0}^3 \Big) \right]\subset \Z_{\geq 0}^3,
$$
where $\operatorname{CH}$ stands for  convex hull, is called the Newton polyhedron associated to $F$.

We will also use extensively the $k$-level of $\Gamma (F)$, that is, the Newton polyhedron associated to $a_k(X,Y)$
$$
\Gamma[k] (F) = \operatorname{CH} \left[ \bigcup_{(i,j) \; | \; (i,j,k) \in N(F)} \Big( (i,j) + \Z_{\geq 0}^2 \Big) \right]\subset \Z_{\geq 0}^2,
$$
as many of our results will look at the evolution of a particular coefficient $a_k(X,Y)$. 
\end{definition}

One should mind that, in general, $\Gamma[k] (F) \neq \Gamma (F) \cap \{Z = k\}$. 

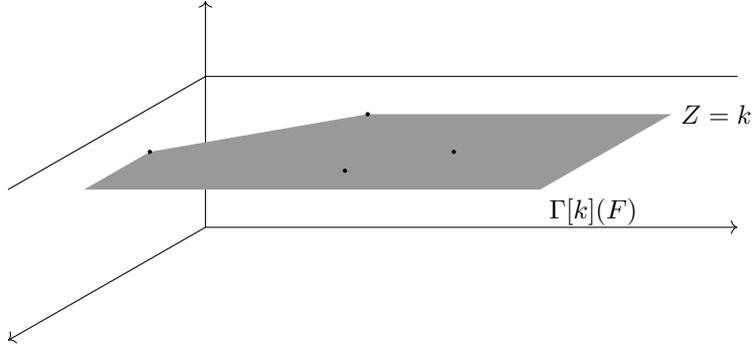
\begin{figure}[hbtp]
    \begin{tikzpicture}[]
        % \draw[pink] (0,0) rectangle (4cm,4cm);

        \begin{scope}[x = {({sin(-60)*1cm},{-cos(60)*1cm})},
                        y = {(1cm,0cm)},
                        z = {(0cm,1cm)}]
            \tikzset{
                every point/.style = {circle, inner sep={4.75\pgflinewidth}, 
                    opacity=1, draw, solid, fill
                },
                point/.style={insert path={node[every point, #1]{}}}, point/.default={},
                point name/.style = {insert path={coordinate (#1)}},
                epsi/.style={black!40}
            }
            
            \def\crad{2\pgflinewidth}
            
            \draw[->] (0,0,0)--(0,0,3);
            \draw[->] (0,0,0)--(0,7,0);
            \draw[->] (0,0,0)--(3,0,0);
            
            \fill[epsi] (3,1,2) -- (2,1,2)--(1,3,2)--(1,7,2 )--(3,7,2) node[below right] {\normalcolor $\Gamma[k](F)$} --cycle; 
            \fill (1,3,2) circle (\crad);
            \fill (2,1,2) circle (\crad);
            \draw (3,0,2) -- (0,0,2) -- (0,7,2);
            \fill (2.5,4,2) circle (\crad);
            \fill (2,5,2) circle (\crad);
            % \fill (1,1,.5) circle (\crad);
            
            \node[right] at (1,7,2) {$Z=k$};
            
            % \draw[very thick] (0,0,0) -- (1,7,2);
            % \fill[] (0, 0, 2) circle (\crad);

         \end{scope}
    \end{tikzpicture}
    \caption{The $k$-level of $\Gamma(F)$.}
\end{figure}

Hironaka's actual version of the Newton polyhedron was a projection of $\Gamma(F)$:

\begin{definition}
The Hironaka polygon of $F$ is
\[
\Delta (F) = \operatorname{CH} \left( \bigcup_{a_{ijk}\neq 0} \biggl[  \biggl( \frac{i}{n-k}, \frac{j}{n-k} \biggr) + \Q_{\geq 0}^2 \biggr]  \right) \subset \Q_{\geq0}^2,
\]
\end{definition}

This object appeared for the first time in the famous Bowdoin Lectures~\cite{Bowdoin}. Note that $\Delta (F)$ can also be read in the following way: let $\rho$ denote the mapping
$$
\begin{array}{rcccl}
\rho: N(F) \setminus \bigl\{(0,0,n)\bigr\} & \longrightarrow & \{Z=0\} \subset \Q^3_{\geq0} & \stackrel{\sim}{\longrightarrow} & \Q_{\geq 0}^2 \\ \\
(i,j,k) & \longmapsto & \displaystyle \left( \frac{i}{n-k}, \frac{j}{n-k},0 \right) & \longmapsto &  \displaystyle \left( \frac{i}{n-k}, \frac{j}{n-k} \right) 
\end{array}
$$

That is, $\rho$ corresponds to a projection from $(0,0,n)$ to the plane $Z=0$, followed by a scaling centered in $(0,0)$ of ratio $1/n$, and
$$
\Delta (F) = \operatorname{CH} \left( \bigcup_{(i,j,k) \in N(F) \setminus \{ (0,0,n) \}} \left[ \rho (i,j,k) + \Q_{\geq 0}^2  \right] \right).
$$

\begin{figure}[hbtp]
    \begin{tikzpicture}[]
        % \draw[pink] (0,0) rectangle (4cm,4cm);

        \begin{scope}[x = {({sin(-60)*.85cm},{-cos(60)*.85cm})},
                        y = {(1cm,0cm)},
                        z = {(0cm,1cm)},
                        scale = 1]
            \tikzset{
                every point/.style = {circle, inner sep={1.75\pgflinewidth}, 
                    opacity=1, draw, solid, fill
                },
                point/.style={insert path={node[every point, #1]{}}}, point/.default={},
                point name/.style = {insert path={coordinate (#1)}},
                epsi/.style={black!40}
            }
            
            \def\crad{1.75\pgflinewidth}
            
            \draw[->] (0,0,0)--(0,0,6);
            \draw[->] (0,0,0)--(0,3,0);
            \draw[->] (0,0,0)--(5,0,0);

            % \fill[black!20] (.5,.25) --  (.2,.6) -- (.2,3) -- (3,3) -- (3,.25) -- cycle;
			\fill [black!20] (5,.5) -- (1,.5) -- (.6,.6)-- (.6,3) -- (5,3)--cycle;

            % \draw (0,0,5) -- (2,1,1) --(5/2, 5/4, 0);
            % \draw [dashed,->] (5/2, 5/4, 0) -- (.5,.25, 0);
            
            \draw (0,0,5) -- (3,3,0);
            \draw [dashed, ->](3,3,0)--(0.600000,0.600000);
            
            \draw (0,0,5) [point];
            \draw (3,3,0) [point];
			\draw (2,1,3) [point];			
			\draw (5,2.5,0) [point];
			\draw (1,0.5,0) [point];
			\draw[->] (0,0,5) -- (5,2.5,0);
			\draw[->, dashed] (5,2.5,0) -- (1,.5,0);
			
            % \draw (2,1,1) [point];
            % \draw (5/2,5/4,0) [point];
            \draw (.6,.6) [point];
            % \draw (.5,.25) [point];

            \node [left] at (0,0,5) {$(0,0,5)$};
            % \node [left] at (2,1,1) {$(2,1,1)$};
            \node [below] at (5,2.5,0) {$(5,5/2,0)$};
			\node [left] at (2,1,3) {$(2,1,3)$};
            \node [right] at (3,3,0) {$(3,3,0)$};
            \node [right] at (3,3,3) {$\rho(3,3,0)=(3/5,3/5,0)$};
            \draw [->,epsi] (3,3,3) to [out=180, in=90] (.6,.6,.1);
            % \node [left] at (2,0,3) {$\rho(2,1,1)=(1/2,1/4)$};
            % \draw[->, epsi] (2,0,3) to[out=0,in=90] (.5,.25,.1);
			\node [left] (kk) at (1,-.5,1.5) {$\rho(2,1,3)=(1,1/2,0)$};
			\draw[->, epsi] (kk.east) to[out=0, in=90] (1,.5,0);

         \end{scope}
    \end{tikzpicture}
    \caption{Projection and scaling of $N(Z^{5}+X^{2}YZ^{3}+X^{3}Y^{3})$.}
    \label{fig:3}
\end{figure}
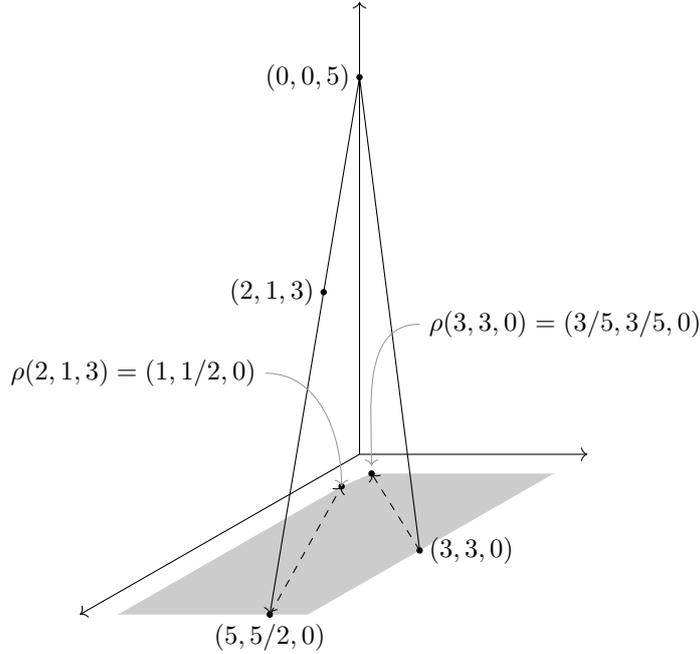

For the rest of the paper, we will identify without mention points $(a,b,0)$ with points $(a,b)$. Also, if $\rho(i,j,k)=(a,b)$, we will say that $(a,b)$ represents the point $(i,j,k)$ or, abusing notation, that it represents the monomial $X^{i}Y^{j}Z^{k}$ (see Figure~\ref{fig:3}).

The combinatorial objects defined above, $N(F)$, $ \Gamma(F)$ and especially the Hironaka polygon $\Delta(F)$, contain useful information about the surface $\cS$. For instance, 
\begin{itemize}
\item Both in $N(F)$ and in $\Gamma(F)$ the monomials of $\overline{F}$ are represented by the points appearing in the plane $x+y+z=n$.
\item In $\Delta(F)$ the monomials of $\overline{F}$ other than $Z^n$ are represented by points on the line $x+y=1$ (see Figure~\ref{fig:4}).
\end{itemize}
Moreover we can state by means of the polygon $\Delta(F)$ if the curves $(Z,X)$ and $(Z,Y)$ are permissible, as we will make explicit in Section \ref{subsec:monoidal} (this is not true for general permissible curves). For the purpose of Section \ref{Sec6} it will be crucial to observe that if we blow up a surface of multiplicity $n$, the multiplicity drops the moment that there exists a point $(a,b)\in\Delta(F^{(r)})$ with $a+b<1$. But, as pointed out in \cite{LZ0} the Hironaka polygon does not contain enough information to keep track of the resolution process (see also Example \ref{ejemplito}).

\begin{figure}[htbp]
    \begin{tikzpicture}[scale=2]
    \tikzset{
        every point/.style = {circle, inner sep={2\pgflinewidth}, 
            opacity=1, draw, solid, fill
        },
        point/.style={insert path={node[every point, #1]{}}}, point/.default={},
        point name/.style = {insert path={coordinate (#1)}},
    }
    \fill[black!20] (2.333333,0.000000)  -- (1.000000,0.000000)  -- (0.000000,1.000000)  -- (0.000000,2.000000)  -- (2.333333,2.000000)  -- cycle;
    \draw[->] (0.000000,0.000000)  -- (2.333333,0.000000)  ;
    \draw[->] (0.000000,0.000000)  -- (0.000000,2.000000)  ;
    \draw[very thin] (0,0) -- (0,-3pt) node[below] {$0$};
    \draw[very thin] (1,0) -- (1,-3pt) node[below] {$1$};
    \draw[very thin] (0,1) -- (-3pt,1) node[left] {$1$}; 
        \draw (0.000000,1.000000)  [point];
        \draw (1.000000,0.000000)  [point];
        \draw (1.333333,0.000000)  [point];
        \draw[shorten >=-20pt, shorten <=-20pt] (1,0) -- (0,1);
        
    \node[above right] at(0,1) {$\rho(0,3,0)$};
    \node[right] at ([yshift=.5cm] 1,0) {$\rho(2,0,1)$};
    \draw[->,, shorten >=2pt] (1,.5) to[out=180, in=90] (1,0);
    \node[above right] at(1.33,0) {$\rho(4,0,0)$};

    \end{tikzpicture}
    \caption{$\Delta(Z^3+X^2Z+Y^3-X^4)$. The monomials of the initial form of $F$ other than $Z^{3}$ lay over the line $x+y=1$.}
    \label{fig:4}
\end{figure}
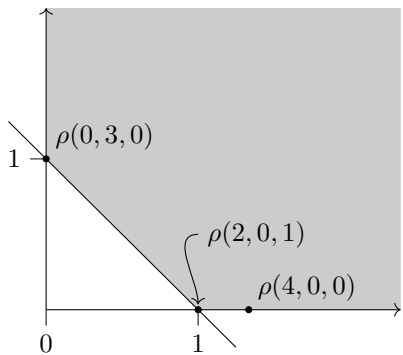

Obviously, even if we have a Weierstrass equation, the equation, and therefore the combinatorial objects defined above, depends on the choice of coordinates. We will omit this dependence as it will usually be clear. 

The next step will be to find an optimal choice of coordinates, so to speak. For instance, if we allow $Z$ to vary by changes of variable of the type
$$
Z  \longmapsto Z + \alpha(X, Y), \mbox{ with } \alpha \in K[[X,Y]] \mbox{ not a unit},
$$
we obtain a collection of Hironaka polygons which have a minimal element in the sense of inclusion. This object was called by Hironaka \cite{H1} the {\em characteristic polygon} of the pair $\bigl(\cS, \{ X,Y\}\bigr)$ and was denoted by $\Delta \bigl(\cS,\{X,Y\}\bigr)$.

\begin{definition}
A vertex $(P_1,P_2)$ of $\Delta (F)$ is called contractible if there exists a change of variables $\varphi$
$$
Z  \longmapsto  Z + \alpha X^a Y^b, \qquad\mbox{ with } \alpha \in K,
$$
such that
$$
\Delta \bigl(\varphi(F)\bigr) \subset \Delta (F) \setminus \bigl\{(P_1,P_2) \bigr\}.
$$
In this case, $\varphi$ is called a contraction of the vertex $(P_1,P_2)$.
\end{definition}

After applying to $F$ the Tchirnhausen transformation
$$
Z \longmapsto  Z - \frac{1}{n} a_{n-1}(X,Y),
$$
the resulting equation has no contractible vertices in its Hironaka polygon. In fact, a vertex $(a,b)$ is contractible if and only if it represents {\em all} the monomials from $(Z + \alpha X^a Y^b)^n$ and this cannot happen since, after applying the Tchirnhausen transformation,~$a_{n-1}(X,Y)=0$. As it becomes obvious from the equations associated to the different blow-ups, the coefficient~$a_{n-1}(X,Y)$~will remain~null during the resolution process. In classical terms, $Z=0$ is a hyperplane with permanent maximal contact with the surface $\cS$. This is one of the places where characteristic zero plays a role, as these assertions are no longer true in positive characteristic.

Hironaka proved in \cite{H1} (for arbitrary characteristic) that the vertices of $\Delta (F)$ are not contractible if and only if $\Delta (F) = \Delta \bigl(\cS, \{X,Y\}\bigr)$. From the previous discussion this is obvious in characteristic zero.

\begin{definition}
A Weierstrass-Tchirnhausen ({\WT}) equation of $\cS$ is an equation of the form
$$
	F(X,Y,Z) = Z^n + \sum_{k=0}^{n-2} \Biggl(\sum_{i+j\geq n-k} a_{ijk}X^iY^j \Biggr) Z^k \in K[[X,Y]][Z].
$$
\end{definition}

Most of the results of the paper depend on the form of the tangent cone, more precisely we have to distinguish between plane and non-plane tangent cone. Notice that $\cC$ being a plane is equivalent to $\overline{F}=Z^n$, since $F$ is a {\WT} equation. 

\begin{definition}
For a {\WT} equation $F$, we will denote by 
$$
\begin{aligned}
\cR(F)&=\bigl(R_1(F),R_2(F)\bigr)&&\text{ the rightmost vertex of }\Delta(F),\\
\cL(F)&=\bigl(L_1(F),L_2(F)\bigr)&&\text{ the leftmost vertex of  }\Delta(F),
\end{aligned}
$$
which will turn out to be the most interesting points of the polygon. For a given level~$\Gamma[k]$ of the Newton polyhedron we will speak accordingly of $\cR[k]$ and/or $\cL[k]$.
\end{definition}

A major invariant for the control of the resolution process will be $L_2$ (and $L[k]_2$, at $\Gamma[k]$). This number has been used extensively in the literature since \cite{Bowdoin}, but it does not lead to effective bounding results unless some precautions are taken.

\begin{example}\label{nhpolygon}
    It is worth noting that 
Hironaka's polygon does not provide a  description of the equation, even for {\WT}~equations. For instance, it is possible for different points of $N (F)$ to be identified by means of $\rho$.  What is more, in general it is not possible to get rid of all of these identifications by changing  variables in $K[[X,Y]]$. An easy example is the equation
$$
F = Z^4 + (Y-X)^4Z^2 + (Y+3X)^8.
$$
\end{example}

We found this fact to be one of the main obstructions when  bounding the number of required blow-ups. Our idea has been to limit the choice of variables, fixing a special type of Weierstrass equation, in exchange for a greater control in the combinatorics.

Another possible way to tackle this problem is to change the combinatorial object, and in \cite{CST}, we introduced a modified version of the Newton-Hironaka polytope that better tracks the points of~$N(F)$. In this paper, however, we explore bounding the number of blow-ups by preparing the equations.

\subsection{GWT Equations}
\label{S3}

Fix a {\WT} equation
$$
F(X,Y,Z) = Z^n + \sum_{k=0}^{n-2} a_k(X,Y) Z^k \in K[[X,Y]][Z],
$$
and for each coefficient $a_k(X,Y)$ consider its initial form $\overline{a_k}(X,Y)$. Let us write $\nu_k = \nu(a_k) \geq n-k$. 
Paralleling the argument for Weierstrass equations, we can find a projective point $(1: \alpha) \in \bP^2(K)$ such that:
$$
\overline{a_k} (1, \alpha) \neq 0, \qquad\mbox{ for all } k=0,\dots,n-2.
$$
Therefore, the linear change of variables
$$
Y \longmapsto Y + \alpha X
$$
takes $F$ into another {\WT} equation, since the {\WT} condition only depends on the choice of $Z$,  where now the coefficients (renamed as $a_k$ again for simplicity) have the form
$$
\overline{a_k}(X,Y) = a_{\nu_k0k} X^{\nu_k} + \mbox{(other terms)}, \qquad\mbox{ with } a_{\nu_k0k} \neq 0.
$$
That is, all coefficients are now regular on $X$.

\begin{definition}
A Generalized Weierstrass-Tchirnhausen ({\GWT}) equation of $\cS$ is an equation of the form
$$
F(X,Y,Z) = Z^n + \sum_{k=0}^{n-2} a_k (X,Y) Z^k \in K[[X,Y]][Z],
$$
such that all the $a_k (X,Y)$ verify:
\begin{itemize}
\item $\nu (a_k) = \nu_k \geq n-k$.
\item $a_k (X,Y)$ is regular in $X$ with order $\nu_k$.
\end{itemize}
\end{definition}

As we have just seen, such an equation always exists, although there are, in principle, infinitely many {\GWT} equations for a given surface. Note that we have privileged a variable ($X$ in our case, although the choice of $Y$ is obviously analogous).

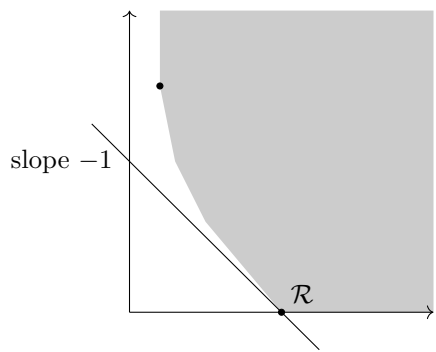
\begin{figure}[htbp]
    \begin{tikzpicture}[scale=2]
    \tikzset{
        every point/.style = {circle, inner sep={2\pgflinewidth}, 
            opacity=1, draw, solid, fill
        },
        point/.style={insert path={node[every point, #1]{}}}, point/.default={},
        point name/.style = {insert path={coordinate (#1)}},
    }
 
    \fill[black!20] (.2,2) -- (.2,1.5) -- 
        (.3, 1.) --
        (.5, .6) --
    (1,0) -- (2,0) --(2,2)--cycle ;
    \draw[->] (0.000000,0.000000)  -- (2,0.000000)  ;
    \draw[->] (0.000000,0.000000)  -- (0.000000,2.000000)  ;
    \draw (1,0) [point];
    \draw (.2,1.5) [point];
    \draw[shorten >=-20pt, shorten <=-20pt] (1,0) -- (0,1) node[below, left] {slope $-1 \;$};
    
    \node[above right] at (1,0) {$\mathcal{R}$};
    \end{tikzpicture}
    \caption{Typical shape of the polyhedron $\Delta(F)$, when $F$ is {\GWT}.}
	\label{fig:5}
\end{figure}

From a combinatorial point of view, a {\GWT} equation is characterized by the following fact: every $\Gamma[k](F)$ has a point in the $OX$-axis. Therefore $\cR[k]$ must lie in this axis. And in fact, if $\Gamma[k](F)$ is not a quadrant, then the slope of the compact edge which includes $\cR[k]$ must be $-1$ or smaller (see Figure~\ref{fig:5}). The same holds for the polygon $\Delta(F)$.

\section{Transforming the Equations: Blow-ups and Transvections}
\label{sec:3}

The rest of the paper is devoted to study to which extent $\Delta (F)$ or $\Gamma (F)$, for a {\GWT} equation $F$, allow us to control the resolution process and, as a consequence, to bound the number of blow-ups needed before the multiplicity drops. 
%From this section on out until Section~\ref{Sec6}, we will often make the additional assumption that the tangent cone of~$F$ is a plane, more precisely, that~$\overline{F}=Z^{n}$. Most of the results are also valid in some form for a general tangent cone, but this case can be dealt with easily (see Section \ref{sec:I}).

\subsection{Blow-ups}
\label{Bl}

We can easily rewrite the blow-ups of smooth equimultiple centers as transformations of the equation $F(X,Y,Z)$. In this case, we will write:

\begin{itemize}
\item $\pi^\fm_{(a:b:c)}$ for the quadratic transformation in the direction $(a:b:c)$ of the exceptional divisor.
\item $\varpi^\fm_{(a:b:c)}$ for the associated ring homomorphism, described by (say $a\neq 0$):
$$
X \longmapsto X, \quad Y \longmapsto X \biggl( Y + \frac{b}{a} \biggr), \quad Z \longmapsto X \biggl( Z + \frac{c}{a} \biggr).
$$
Then, $F^{(1)} = \varpi^\fm_{(a:b:c)}(F)/X^n$ is an equation of the strict transform of the blown up surface $\cS^{(1)} = \pi^\fm_{(a:b:c)} (\cS)$. We will call the above transformation the {\em equation of the blow-up}, a common abuse of notation.
\end{itemize}

If we blow up a permissible curve $\fp$, the corresponding situation may not be so clear, but it is nevertheless easy. 

Assume, for instance, $\fp = (Z,X)$ is an equimultiple curve (recall that our procedure only blows up curves of the form $(Z,X)$ and $(Z,Y)$). Then the situation parallels the quadratic transformation, as one can define the associated ring homomorphism $\varpi^\fp_{(a:b:c)}$ on the direction $(a:b:c)$ of the exceptional divisor by:
$$
X \longmapsto X, \quad Y \longmapsto Y, \quad Z \longmapsto X \biggl( Z + \frac{c}{a} \biggr),
$$
and then $F^{(1)} = \varpi_{(a:b:c)}^\fp(F)/X^n$ is an equation of $\cS^{(1)} = \pi_{(a:b:c)}^\fp (\cS)$.

Although it will not be used later, we can sketch the situation for a more general permissible center. If we have a permissible curve $\fp = (G_1,G_2)$, since $F \in \fp^n$ we can assume $G_1 = Z$, because $F$ is a {\GWT} equation (in fact, {\WT} is enough). In the same way, if
$$
\overline{G_2} = Y + \alpha X
$$
(respectively $\overline{G_2} = X + \alpha Y$), the Weierstrass Preparation Theorem allows us to write
$$
\fp = \bigl(Z,Y+X^sv(X)\bigr), \qquad \left( \mbox{resp. } \fp = \bigl(Z, X + Y^s v(Y)\bigr) \right)
$$
with $v(X) \in K[[X]]^*$ and $s\geq 1$. Hence, if $\fp = \bigl(Z,Y+X^sv(X)\bigr)$ then, because of equimultiplicity, it must hold
$$
\bigl(Y+X^sv(X)\bigr)^{n-k} \divides a_k (X,Y), \qquad k=0,\dots,n-2;
$$
and an equation for the strict transform of the monoidal transformation $\cS^{(1)}=\pi^{\fp}(\cS)$ of $\cS$ is
$$
F^{(1)} = Z^n + \sum_{k=0}^{n-2} \frac{a_k(X,Y)}{\bigl(Y+X^sv(X)\bigr)^{n-k}} Z^k.
$$

\subsection{Transvections} 

Our study of how blow-ups alter the combinatorial objects defined above will need us to understand how changes of variables affect them. One must be careful now with coordinate changes, even linear ones, as we may lose the {\GWT} condition. Nevertheless, if the changes only affect $\{X,Y\}$ we can readily assure that, at least, {\WT} equations will be transformed into {\WT} equations.

For a number of reasons that will become clear later, the only changes we will be interested in are of a special type.

\begin{definition}
    A transvection is a change of coordinates in $K[[X,Y,Z]]$ (or simply in $K[[X,Y]]$, as the case may be) given by
    $$
    X \longmapsto X, \quad
    Y  \longmapsto  Y + \sum_{i\geq 1} \alpha_i X^i, \quad Z\longmapsto Z, \qquad\mbox{ with } \alpha_i \in K.
    $$
\end{definition}

    Note that, in particular, the change of variables  needed to go from a {\WT} equation to a {\GWT} is a transvection.

As a transvection affects monomials in the following way:
$$
X^iY^jZ^k  \longmapsto  X^i \biggl( Y + \sum_{l\geq 1} \alpha_l X^l \biggr)^j Z^k, 
$$
a simple calculation tells us that a point $\rho (i,j,k)$~in $\Delta(F)$ is ``transformed'' into a possibly infinite set of points (see Figure~\ref{manypoints}). Indeed, for each $l \neq 0$, we have the points
$$
\left\{  \left( \frac{i}{n-k}, \frac{j}{n-k} \right), 
\left( \frac{i+l}{n-k}, \frac{j-1}{n-k} \right),\dots, 
\left( \frac{i+jl}{n-k}, 0 \right)  \right\}.
$$
Moreover,  these points lie in a segment starting on the point corresponding to $X^iY^jZ^k$ with slope $-1/l$, plus some other points on the $OX$-axis.

\begin{figure}
	% \missingfigure{Aquí el dibujo de los transformados.}
    \tikzset{
        every point/.style = {circle, inner sep={2\pgflinewidth}, 
            opacity=1, draw, solid, fill
        },
        point/.style={insert path={node[every point, #1]{}}}, point/.default={},
        point name/.style = {insert path={coordinate (#1)}},
    }
	\begin{tikzpicture}
		\draw[->] (0,0)--(0,2);
		\draw[->] (0,0)--(6,0);
		
		\begin{scope}[fill=gray,draw=gray]
			\draw (1.000000,1.500000) --(2.500000,0.000000) ;
			\draw (1.000000,1.500000) [point];
			\draw (1.500000,1.000000) [point];
			\draw (2.000000,0.500000) [point];
			\draw (2.500000,0.000000) [point];
			\draw (1.000000,1.500000) --(4.000000,0.000000) ;
			\draw (1.000000,1.500000) [point];
			\draw (2.000000,1.000000) [point];
			\draw (3.000000,0.500000) [point];
			\draw (4.000000,0.000000) [point];
			\draw (1.000000,1.500000) --(5.500000,0.000000) ;
			\draw (1.000000,1.500000) [point];
			\draw (2.500000,1.000000) [point];
			\draw (4.000000,0.500000) [point];
			\draw (5.500000,0.000000) [point];
		\end{scope}
		
		\draw (1.000000,1.500000) [point];
		\node[above right] at (1.000000,1.500000) {$\rho(i,3,k)$};
		\node at (5,1) {$\dots$};
	\end{tikzpicture}
	\caption{A point $\rho(i,j,k)$ ``transforms'' into a possible infinite set of points by a transvection. Here we draw the case $j=3$ and $\alpha_1,\alpha_2,\alpha_3\neq 0$.}\label{manypoints}
\end{figure}
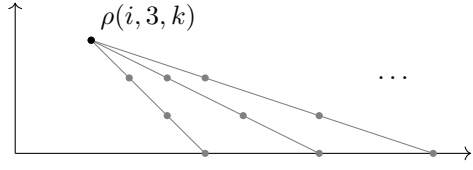

\begin{remark}
    It is clear that if $G$~is the result of applying a transvection to an equation~$F$,  then $L_2$  is invariant in the sense that $\mathcal{L}\bigl(F\bigr)$ and $\mathcal{L}\bigl(G\bigr)$ have the second coordinate equal. Similarly for  $L[k]_2$  in~$\Gamma[k](F)$.
\label{remT}
\end{remark}

\subsection{Factoring the Blow-ups} 

Transvections and blow-ups share a very interesting relationship. 
We will see how  transvections and what we  call combinatorial transformations are the building bricks of the chain of blow-ups. 

\begin{definition}
A combinatorial transformation is either a monoidal transformation centered on $(Z,X)$ or $(Z,Y)$, should these centers be permissible, or a quadratic transformation on $(1:0:0)$ or $(0:1:0)$.
\end{definition}

It is straightforward to show that essentially, a quadratic transformation can be understood as the composition of a transvection (eventually the identity) and a combinatorial transformation.

\begin{lemma} The following diagram is commutative
    \[
    \begin{tikzcd}[row sep=huge, column sep=huge]
    K[[X,Y,Z]] 
        \arrow[r,"\varphi"]
        \arrow[d,"\varpi^{\mathfrak{m}}_{(1\colon \alpha\colon 0)}"']
    & K[[X,Y,Z]]\arrow[dl, "\varpi^{\mathfrak{m}}_{(1:0:0)}"]\\
    K[[X,Y,Z]]
\end{tikzcd}
\]
with $\varphi$ a transvection given by $Y \longmapsto Y + \alpha X$.
\end{lemma}

 It will also be interesting to consider the situation where several of these transformations are factored through transvections, and the result is rather similar.

\begin{lemma}
The following diagram is commutative,
\[
\begin{tikzcd}[cramped, row sep=huge, column sep=large ]
    K[[X,Y,Z]] \arrow[r, "\varpi^{\mathfrak{m}}_{(1:0:0)}"] &
    K[[X,Y,Z]] \arrow[r, ] &[-2.5em] % reduce middle arrow
    \cdots \arrow[r, ] &[-2.5em]     % reduce middle arrow
    K[[X,Y,Z]] \arrow[r, "\varpi^{\mathfrak{m}}_{(1:0:0)}"] &
    K[[X,Y,Z]]  \\
    K[[X,Y,Z]] \arrow[r, "\varpi^{\mathfrak{m}}_{(1:\alpha_{1}:0)}"] \arrow[u, "\varphi_{1}"]
    &
    K[[X,Y,Z]] \arrow[r, ] \arrow[u, "\varphi_{2}"]&
    \cdots \arrow[r, ] &
    K[[X,Y,Z]] \arrow[r, "\varpi^{\mathfrak{m}}_{(1:\alpha_{m}:0)}"] \arrow[u, "\varphi_{m}"]&
    K[[X,Y,Z]]  \arrow[u, "\operatorname{id}"]
\end{tikzcd}
\]
\noindent where $\varphi_i$ is the transvection defined by
$$
\varphi_i  : Y \longmapsto  Y + \alpha_i X + \dots + \alpha_m X^{m-i+1}.
$$
\end{lemma}

A lemma that will be useful later on is the following, which explains how quadratic transformations relate to each other via transvections.

\begin{lemma} \label{transvectionsquare}
Given a quadratic transformation on some direction $(1:\alpha:0)$ with its associated ring homomorphism $\varpi^\fm_{(1:\alpha:0)}$, and a transvection defined by
$$
\varphi: Y \longmapsto Y + v(X) = Y + \sum_{i=1}^\infty v_iX^i,
$$
there exists a transvection $\psi$ and a direction $(1:\beta:0)$ making the following diagram commutative:%

    \[
    \begin{tikzcd}[row sep=huge, column sep=huge]
    K[[X,Y,Z]] \arrow[r,"\varpi^{\mathfrak{m}}_{(1:\alpha:0)}"]
        \arrow[d,"\varphi"']
    & K[[X,Y,Z]]\arrow[black!60,d,"\psi"]\\
    K[[X,Y,Z]] \arrow[black!60, r, "\varpi^{\mathfrak{m}}_{(1:\beta:0)}"] &
    K[[X,Y,Z]]
\end{tikzcd}
\]
\end{lemma}

\begin{proof}
It is an easy exercise on diagram chasing. Note that
$$
\beta = \alpha - v_1, \qquad \psi: Y \longmapsto Y + \sum_{i=2}^{\infty} v_i X^{i-1}.
\qedhere
$$
\end{proof}

As for monoidal transformations are concerned, if we have a permissible curve $\fp = \bigl(Z,X+Y^sv(Y)\bigr)$, every monoidal transformation can also be split into one
combinatorial transformation and a change of variables in $K[[X,Y]]$, \emph{not necessarily a transvection}. However, following Algorithm~\ref{alg}, we will not need to use these general changes of coordinates.

\subsection{Monoidal Combinatorial Transformations}
\label{subsec:monoidal}

Note that $(Z,X)$ is permissible if and only if $i+k \geq n$ for all $(i,j,k) \in N(F)$. This amounts to $\Delta (F) \subset \{x \geq  1\}$. The case of directions other than $(1:0:0)$ of the exceptional divisor only appears when the tangent cone is not plane and can be easily dealt with, as we will see in Section \ref{sec:I}. 

A monoidal transformation centered in $(Z,X)$ with direction $(1:0:0)$ acts on monomials in the following way:
$$
X^iY^jZ^k  \longmapsto  X^{i+k-n} Y^j Z^k, 
$$
and hence acts on $N(F)$ as follows:
\[
(i,j,k)  \longmapsto  (i+k-n,j,k).
\]
Therefore every point $(i,j,k)$ undergoes a translation of vector $(-n+k,0,0)$ which takes it closer to the $OZ$-axis. Similar statements hold for $\Gamma (F)$ and $\Gamma[k](F)$. For $\Delta(F)$, the transformation induces the following effect on the polygon:
$$
\left( \frac{i}{n-k}, \frac{j}{n-k} \right) \longmapsto
 \left( \frac{i+k-n}{n-k}, \frac{j}{n-k} \right) = \left( \frac{i}{n-k}-1, \frac{j}{n-k} \right).
$$
That is, every point in $\Delta (F)$ undergoes a translation of vector $(-1,0)$. 

In a similar way, a monoidal transformation centered in $(Z,Y)$ with direction $(0:1:0)$ induces a translation of vector $(0,-1)$ on the points of $\Delta (F)$. Note, however, that $(Z,Y)$ is never permissible for a {\GWT} equation.

\subsection{Quadratic Transformations}
\label{ssQT}
We will focus on the case $\overline{F}=Z^n$ since it is the case of study of Section \ref{secquad}. 
As we have already seen, any quadratic transformation, say on $(a:b:0)$, can be split into a transvection of $K[[X,Y]]$ (eventually the identity if $ab=0$) and a combinatorial transformation.

A quadratic transformation, say with direction  $(1:0:0)$, induces the following transformation on monomials:
$$
X^iY^jZ^k  \longmapsto  X^{i+j+k-n} Y^j Z^k, 
$$ 
so, in $\Delta(F)$ it amounts to
$$
\left( \frac{i}{n-k}, \frac{j}{n-k} \right) \longmapsto
 \left( \frac{i+j+k-n}{n-k}, \frac{j}{n-k} \right) = \left( \frac{i+j}{n-k}-1, \frac{j}{n-k} \right).
$$

Therefore each point moves on its horizontal line; and the greater its $y$-coordinate is, the more it moves rightwards. 

\begin{remark}
	Note that, in particular, if $\overline{F}=Z^n$ and $(Z,Y)$ is not permissible, at least we must have one point 
	$$
	\left( \frac{i}{n-k}, \frac{j}{n-k} \right) \in \Delta (F), \qquad \mbox{ with } \frac{j}{n-k} < 1,
	$$
	which will therefore move {\em to the left} (see Figure~\ref{fig:7}).
\end{remark}

\begin{figure}[htbp]
    \begin{tikzpicture}[scale=2]
    \tikzset{
        every point/.style = {circle, inner sep={2\pgflinewidth}, 
            opacity=1, draw, solid, fill
        },
        point/.style={insert path={node[every point, #1]{}}}, point/.default={},
        point name/.style = {insert path={coordinate (#1)}},
    }
 
     \fill[black!20] (.2,2)-- (.4,.4) -- (1,0) -- (2,0) -- (2,2) -- cycle;
 
    \draw[->] (0.000000,0.000000)  -- (2,0.000000)  ;
    \draw[->] (0.000000,0.000000)  -- (0.000000,2.000000)  ;
    
    \draw[very thick] (0,1) -- (2,1) node[right] {$y=1$};

    \draw (1,1.5) [point, point name=A];
    \draw (1.5,1.5) [point={outer sep=1cm}, point name = B]; \draw [->] (A) -- (B);
    
    \draw (1.25,.5) [point, point name = C];
    \draw (.75,.5) [point, point name = D];
    \draw [->] (C)--(D);
    
    \end{tikzpicture}
    \caption{Relative movements of quadratic combinatorial transformations.}
	\label{fig:7}
\end{figure}
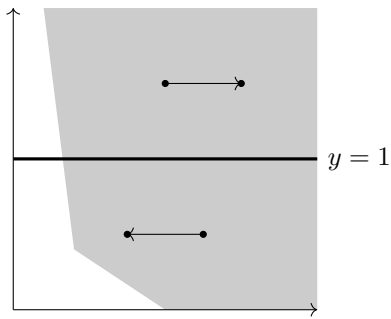

Taking the direction $(1:\alpha:0)$  each point is ``transformed'' into an array of points, all of them into the same vertical line. More precisely, over monomials, the transformation goes as follows:
$$
X^iY^jZ^k  \longmapsto  X^i \left( Y + \alpha X \right)^j Z^k  \longmapsto  X^{i+j+k-n} \left( Y + \alpha \right)^j Z^k;
$$
while in $\Gamma[k](F)$ goes as:
$$
(i,j,k)  \longmapsto  (i+j+k-n, j-l,k), \qquad l=0,\dots,j;
$$
and therefore in $\Delta (F)$ (see Figure~\ref{fig:transform}):
$$
\left( \frac{i}{n-k},\frac{j}{n-k} \right)  \longmapsto  \left( \frac{i+j}{n-k}-1, \frac{j-l}{n-k} \right), \qquad l=0,\dots,j. 
$$

\begin{figure}[htbp]
    \begin{tikzpicture}[scale=1]
    \tikzset{
        every point/.style = {circle, inner sep={2\pgflinewidth}, 
            opacity=1, draw, solid, fill
        },
        point/.style={insert path={node[every point, #1]{}}}, point/.default={},
        point name/.style = {insert path={node[circle,inner sep={2\pgflinewidth}] (#1){}}},
        aux/.style = {black!20}
    }
 
    \draw[->] (0,0) -- (6,0);
    \draw[->] (0,0) -- (0,6);
    \draw[aux] (0,5) -- (5,0);
    \draw[aux] (4,0) -- (4,6);
    \draw (1,4) [point, point name=A];
    \foreach \s in {0,...,4}
        {
            % \draw[aux,->] (1,4) to[bend left] (4,\s);
            \draw (4,\s) [point, point name={B\s}];
        }
    \begin{scope}[aux,->]
        \draw (A) to[bend right] (B0);
        \draw (A) to[bend right] (B1);
        \draw (A) to[bend left] (B2);
        \draw (A) to[bend left] (B3);
        \draw (A) to[bend left] (B4);
    \end{scope}
    \node [above] at (A) {$(a,b)$};
    
    \draw  [decorate,decoration={brace}, yshift=-1pt] (5,0) -- (4,0) node[midway, below] {$1$};
         
    \node[right] (L) at (5,1) {$x+y=a+b$};
    \draw[->, aux] (L.west) to[out=180, in=45] ($(4,1)!.5!(5,0)$) ;
         
    \end{tikzpicture}
    \caption{``Transform'' under $\varpi_{(1:\alpha:0)}^{\mathfrak{m}}$ of the point representing a monomial $X^{i}Y^{j}Z^{k}$ .}
	\label{fig:transform}
\end{figure}
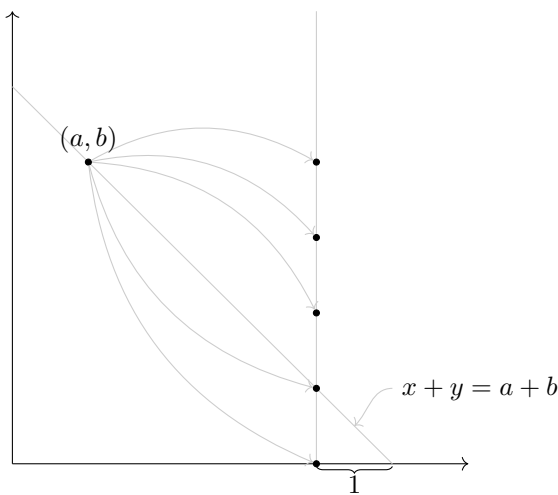

\section{Evolution of the polygon under blow-ups: towards a quadrant}
\label{secquad}
In this section we will consider a {\WT} equation $F$ with plane tangent cone. 
We present some partial results studying the evolution of the polygon $\Delta(F)$ under blow-ups (see Section \ref{sec:I} for a treatment of the non-plane tangent cone case). We will also give an example  illustrating why it is not possible to solve this problem in full generality. 

Note that, once the multiplicity drops, the resulting equation will {\em not} be a Weierstrass equation, so all combinatorics will have to be computed again. Therefore it makes sense to focus on 
the process up to and until there is a decrease in the multiplicity.

Consequently, when we consider the polygon $\Delta(F)$ and try to describe $\Delta(F^{(1)})$ where $F^{(1)}$ is the transform after a blow-up, we are implicity assuming that there is no decrease in multiplicity.

Quadrants are a particularly comfortable type of polygon to work with. As we know they may be characterized as:
\begin{itemize}
\item Polygons with only one vertex (that is, $\cL = \cR$).
\item Polygons without compact faces of positive dimension.
\item Polygons for which there is a point which is minimal for both $x$-coordinate and $y$-coordinate.
\end{itemize}
For consistency's sake, we will consider the empty set as a quadrant.

Problem 6 of \cite{HS} is concerned with the relationship between {\em improving} the singularity (in some ample sense) and the fact that the combinatorics associated are {\em closer} to the quadrant case. This section  partially answers Problem 6 in the sense that we prove that for almost all directions of the exceptional divisor, the polygon $\Delta(F)$ tends to a quadrant.

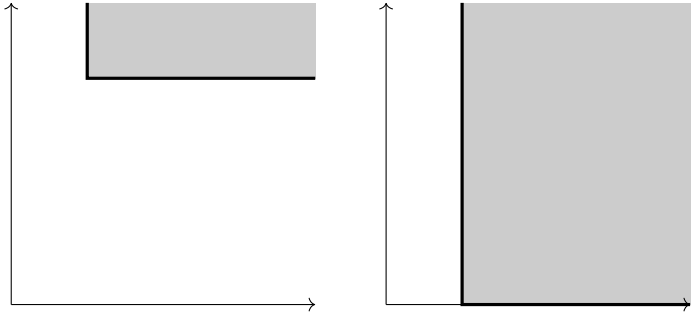
\begin{figure}[htbp]
    \begin{tikzpicture}[scale=1]
    \tikzset{
        every point/.style = {circle, inner sep={2\pgflinewidth}, 
            opacity=1, draw, solid, fill
        },
        point/.style={insert path={node[every point, #1]{}}}, point/.default={},
        point name/.style = {insert path={node[circle,inner sep={2\pgflinewidth}] (#1){}}},
        aux/.style = {black!20}
    }
 
    \draw[->] (0,0) -- (4,0);
    \draw[->] (0,0) -- (0,4);
    \fill[aux] (1,3) rectangle (4,4);
    \draw [very thick] (1,4) -- (1,3) --(4,3);
    \end{tikzpicture}
    \qquad 
    \begin{tikzpicture}[scale=1]
    \tikzset{
        every point/.style = {circle, inner sep={2\pgflinewidth}, 
            opacity=1, draw, solid, fill
        },
        point/.style={insert path={node[every point, #1]{}}}, point/.default={},
        point name/.style = {insert path={node[circle,inner sep={2\pgflinewidth}] (#1){}}},
        aux/.style = {black!20}
    }
 
    \fill[aux] (1,0) rectangle (4,4);
    \draw [very thick] (1,4) -- (1,0) --(4,0);
    \draw[->] (0,0) -- (4,0);
    \draw[->] (0,0) -- (0,4);
    \end{tikzpicture}
    \caption{A quadrant (left) and a {\GWT} quadrant (right).}
	\label{fig:quadrant}
\end{figure}

We will distinguish a certain type of quadrants, tied to {\GWT} equations.
Let $F$ be {\GWT} equation such that  $\Delta(F)$ is  a quadrant. We call this situation a {\GWT} quadrant (see Figure~\ref{fig:quadrant}).  

\vspace{3mm}

Now we study the evolution of the polygon $\Delta(F)$. Recall first that monoidal transformations induce just translations of the polygon. 
Hence we only have to study the effect of quadratic transformations on the Hironaka polygon. This effect depends on the direction $(a:b:0)$ in the tangent cone. 

First we consider the directions $(1:0:0)$ and $(0:1:0)$, where the evolution of the polygon is easily understood.

\begin{lemma}
Let $\tau$ be a compact facet of the Hironaka polygon $\Delta(F)$ of a {\WT} equation, and let $\alpha$ be the angle that Span$(\tau)$ forms with the positive $x$-axis. The points in $\tau$ are transformed by $\varpi^\fm_{(1:0:0)}$ or $\varpi^\fm_{(0:1:0)}$ in points on a line, hence we can define the compact facet $\tau'$ as the image of $\tau$ by the quadratic transformation, and $\alpha'$ the corresponding angle.
 Then, 
\begin{enumerate}
\item[(i)] if $\tau$ has slope $\leq -1$, its image $\tau'$ by the transformation $\varpi^\fm_{(1:0:0)}$ is not a compact facet of $\Delta(F^{(1)})$. Otherwise, $\tau'$ is a compact facet of $\Delta(F^{(1)})$ with slope smaller than $\tau$. 
%More concretely, if $\tau$ forms an angle $\alpha>\frac{3\pi}{4}$, then its image $\tau'$ forms an angle $\alpha'$ such that
More precisely, if  $\alpha>\frac{3\pi}{4}$, then 
    \[\cotg(\alpha')=\cotg(\alpha)+1.\]

\item[(ii)] If $\tau$ has slope $\geq -1$, its image $\tau'$ by the transformation $\varpi^\fm_{(0:1:0)}$ is not a compact facet of $\Delta(F^{(1)})$. Otherwise, $\tau'$ is a compact facet of $\Delta(F^{(1)})$ with slope bigger than $\tau$. 
%More concretely, if $\tau$ forms an angle $\alpha<\frac{3\pi}{4}$, then its image $\tau'$ forms an angle $\alpha'$ such that
More precisely, if  $\alpha<\frac{3\pi}{4}$, then
    \[\tg(\alpha')=\tg(\alpha)+1.\]
\end{enumerate}
Therefore, the number of compact facets of the Hironaka polygon never increases by quadratic transformations. Moreover, if $\tau$ is a face and $\tau'$ is its image, then $\length(\tau)>\length(\tau')$.
\label{2direcciones}
\end{lemma}

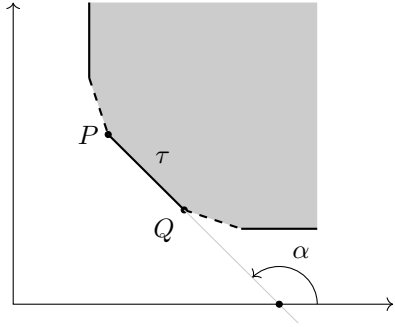
\begin{figure}[htbp]
	\centering
	
	\begin{tikzpicture}
	    \tikzset{
	        every point/.style = {circle, inner sep={2\pgflinewidth}, 
	            opacity=1, draw, solid, fill
	        },
	        point/.style={insert path={node[every point, #1]{}}}, point/.default={},
	        point name/.style = {insert path={node[circle,inner sep={2\pgflinewidth}] (#1){}}},
	        aux/.style = {black!20}
	    }

        \coordinate (R) at (4,0);
        \coordinate (A0) at (1,3);
		\coordinate (A) at (1.25,2.25);
		\coordinate (B) at (2.25,1.25);
        \coordinate (B0) at (3,1); 
        \fill [aux] (4,4) -- (1,4) --(1,3) -- (A) -- (B) --(3,1) -- (4,1) -- cycle;
		\draw [<->, name path=l1] (0,4) -- (0,0) --  (5,0) ; 
		\draw (A)[point];
		\node [left] at (A) {$P$};
		\draw (B)[point];
		\node [below left] at (B) {$Q$};
        \draw[thick] (1,4) -- (1,3);
        \draw[thick] (4,1) -- (3,1);
		\draw[thick, dashed] (1,3) -- (A);
		\draw[thick, dashed] (B) -- (3,1);
        % \node[rotate=-33] at ($(A0)!.5!(A)$) {$\ddots$};
        % \node[rotate=33] at ($(B)!.5!(B0)$) {$\ddots$};
        \draw[aux, name path=l2] (A) -- ($(A)!2.5!(B)$);
        \path[name intersections={of=l1 and l2}];
        \draw (intersection-1) [point];
        \pic [draw, ->, "$\alpha$", angle eccentricity=1.5] {angle = R--intersection-1--A};
        \node [above right] at ($(A)!.5!(B)$) {$\tau$};
		\draw [thick] (A)--(B);
        
	\end{tikzpicture}

	\caption{The angle $\alpha$ measures the evolution of polygons.}
	\label{fig:anguloalfa}
\end{figure}

\begin{proof}
We prove the statement for the transformation $\varpi^\fm_{(1:0:0)}$,   the other case being completely analogous. 

The first claim is straightforward, and we can consider $\tau'$ the image of a facet~$\tau$. What is not obvious is whether $\tau'$ is a facet of $\Delta(F^{(1)})$ or not. The result depends heavily on the slope of $\tau$. Indeed, let $\tau$ be a facet in $\Delta(F)$, and let $P$ (upper) and $Q$ (bottom) its extreme points. Then,
\[
\begin{aligned}
    P&=(P_1,P_2)&&\longmapsto &&&P'&=(P_1+P_2-1,P_2)    \\
    Q&=(Q_1,Q_2)&&\longmapsto &&&Q'&=(Q_1+Q_2-1,Q_2),
\end{aligned}
\]
and the slope of~$\tau$ is $-({P_2-Q_2})\bigm/({Q_1-P_1})$. Consequently, the slope is $\leq -1$ if and only if $P_1+P_2\geq Q_1+Q_2$. 

In this case, looking at the images $P'$ and $Q'$ we have that $Q_1'\leq P_1'$, which implies that the point $P'$, and hence the facet $\tau'$, will no longer be in the border of~$\Delta(F^{(1)})$.

Suppose now that the slope is $>-1$. Then, $P_1+P_2<Q_1+Q_2$. We have
\[
\begin{aligned}
    \cotg(\alpha)&=-\frac{Q_1-P_1}{P_2-Q_2}\\
    \cotg(\alpha')& =-\frac{Q_1'-P_1'}{P_2'-Q_2'}  =-\frac{Q_1+Q_2-1-P_1-P_2+1}{P_2-Q_2} =-\frac{Q_1-P_1}{P_2-Q_2}+1.\\
\end{aligned}
\]

Finally, we look at the length of faces with slope $>-1$ and its corresponding images
\[
\begin{aligned}
    \length(\tau)&=\sqrt{(Q_1-P_1)^2+(P_2-Q_2)^2}\\
    \length(\tau')&=\sqrt{(Q_1+Q_2-1-P_1-P_2+1)^2+(P_2-Q_2)^2},\\
\end{aligned}
\]
and clearly $\length(\tau')<\length(\tau)$.
\end{proof}

\begin{remark}
    Note that the Lemma above claims that if $\tau$ is a compact facet in $\Delta(F)$ of slope~$-1$, then its image $\tau'$ is not a compact facet of $\Delta(F^{(1)})$; but this does not mean that there is no compact facet of slope $-1$ in $\Delta(F^{(1)})$, since it can be the image of a compact facet in $\Delta(F)$ of slope $>-1$ (resp. $<-1$) under the quadratic transformation in the direction $(1:0:0)$ (resp. $(0:1:0)$). Indeed, this is the case of the surface defined by 
    \[F=Z^3-(X^3Y^2+XY^3+Y^4)Z+X^9Y^8,\]
where both $\Delta(F)$ and $\Delta(F^{(1)})$ have facets of slope $-1$.
\end{remark}

\begin{corollary}
After a quadratic transformation in the direction $(1:0:0)$ or $(0:1:0)$, we have that the distance $d(\mathcal L,\mathcal R)$ drops, or in other words
\[d(\mathcal L^{(1)},\mathcal R^{(1)})<d(\mathcal L,\mathcal R).\]
Moreover, if all one-dimensional compact faces of $\Delta(F)$ have slope $\leq -1$ (resp. $\geq -1$), then after the transformation $\varpi^\fm_{(1:0:0)}$ (resp. $\varpi^\fm_{(0:1:0)}$) we have that $\Delta(F^{(1)})$ is a quadrant, though not necessarily a {\GWT} one.
\label{distancia}
\end{corollary}

Now we look at quadratic transformations in the direction $(1:\alpha:0)$, and we have that for most of the directions it is easy to describe $\Delta(F^{(1)})$.

\begin{proposition} \label{almostall}
Let $F$ be a {\WT} equation with $\overline{F} = Z^n$. Assume $\overline{a_k}(1:\alpha) \neq 0$ for all $k=0,\dots,n-2$. Then the equation of the quadratic transformation in the direction $(1:\alpha:0)$ verifies that $\Delta(F^{(1)})$ is a {\GWT} quadrant.
\end{proposition}

\begin{proof}
Let us write
$$
a_k(X,Y) = \overline{a_k} (X,Y) + \sum_{i+j=l>\nu_k} a_{ijk}X^iY^j = \overline{a_k} (X,Y) + \sum_{l>\nu_k} a_{l,k}(X,Y);
$$
that is, we are considering the order of $a_k$ to be $\nu_k$ (in particular, in our case $\nu_k > n-k$). After applying the quadratic transformation, we have the equation
$$
F^{(1)} (Z) = Z^n + \sum_{k=0}^{n-2} \biggl[ X^{k+\nu_k-n} \overline{a_k}(1,Y+\alpha) + \sum_{l>\nu_k} X^{l+k-n} a_{l,k}(1,Y+\alpha) \biggr]Z^k.
$$
Mind the following facts concerning $F^{(1)}$:
\begin{itemize}
\item If we call 
$$
a_k^{(1)} = X^{k+\nu_k-n} \overline{a_k}(1,Y+\alpha) + \sum_{l>\nu_k} X^{l+k-n} a_{l,k}(1,Y+\alpha),
$$
then we have 
$$
\overline{a_k^{(1)}} = X^{k+\nu_k-n} \overline{a_k} (1, \alpha),
$$
since all the remaining monomials have either greater degree in $X$ or the same degree, but with some non-zero exponent in $Y$ as well.
\item The initial form of $a_k^{(1)}$ with respect to $X$ is also
$$
X^{k+\nu_k-n} \overline{a_k} (1, \alpha).
$$
\end{itemize}
These two conditions are equivalent to saying that $\Gamma[k]^{(1)}$ is actually a \GWT{} quadrant, and  so is $\Delta(F)$. 
\end{proof}

\medskip

 As we have just seen,  for {\em almost all} possible directions~$(1:\alpha:0)$ the resulting equation has very nice combinatorial and resolution properties: its characteristic polygon is a \GWT{} quadrant and the multiplicity decreases by simply blowing-up $(Z,X)$  as many times as possible. This suggests that comparatively long resolution processes are an exception, not a rule. 
Obviously, the tricky part is the   {\em almost~all}. And it is because of these (finitely many but extremely annoying) cases that we need to dig deeper into the combinatorial toolbox. This will be done in next section for {\GWT} equations.

The next example illustrates the fact that, for the exceptional  directions $(1:\alpha:0)$, we cannot describe the effect of the corresponding quadratic transformation in terms of  the Hironaka polygon.

\begin{example}
Consider the family of surfaces defined by
\[F=Z^2+(X-Y)^3+X^r,\qquad  \text{with $r>3$.}\]
They all have the same Hironaka polygon $\Delta(F)$. The monomial $X^r$ is hidden in  $\Delta(F)$ but it influences the polygon $\Delta(F^{(1)})$. Indeeed,  let us look at the Hironaka polygon after the quadratic transformation in the direction $(1:1:0)$, then $\Delta(F^{(1)})$ has only one compact facet $\tau$ and
\begin{enumerate}
\item[(i)] if $r=4$, $5$, then the slope of $\tau$ is $<-1$,
\item[(ii)] if $r=6$, then the slope of $\tau$ is equal to $-1$,
\item[(iii)] if $r>6$, then the slope of $\tau$ is $>-1$.
\end{enumerate}
Hence,  we have three different possible behaviors for the same Hironaka polygon $\Delta(F)$.
\label{ejemplito}
\end{example}

We have seen how in most of the cases the polygon evolves to a quadrant, since either we get a quadrant or at least the distance $d(\mathcal L,\mathcal R)$ drops.

%\subsection{Stability}

We are going to explore why quadrants are such nice polygons to work with in the resolution context. A first reason to deal with quadrants is that they are stable, in the following sense.

\begin{proposition}
Assume $\Delta (F)$ is a quadrant, and let $\cS^{(1)}$ be a quadratic transformation of $\cS$. If $F^{(1)}$ is a local {\WT} equation of $\cS^{(1)}$, then $\Delta \left( F^{(1)} \right)$ is also a quadrant.
\label{PSQ}
\end{proposition}

\begin{proof}
It is clear from the properties of {\WT} equations and the results in Section \ref{sec:3} that, under a combinatorial quadratic transformation, if there is a point which is minimal for $x$ and $y$ coordinates, it is transformed into a point with the same property. 
Moreover, we know how a quadratic transformation acts on $\Delta(F)$, as seen in Section~\ref{ssQT}.

Hence the transformation of $\cL = \cR$ will be again a point with minimal $x$ and $y$-coordinate, precisely
$$
\begin{aligned}
\mathcal L^{(1)}&=\bigl(L_1(F)+L_2(F)-1,L_2(F)\bigr), && \mbox{ in the direction }(1:0:0)\\
\mathcal L^{(1)}&=\bigl(L_1(F),L_1(F)+L_2(F)-1\bigr), && \mbox{ in the direction }(0:1:0)\\
\cL^{(1)} &= \bigl( L_1(F) + L_2(F) - 1, 0 \bigr) && \mbox{ in the direction }(1:\alpha:0).\qedhere
\end{aligned}
%\qedhere
$$
%In particular, as the same goes for every $\Gamma[k](F)$, unless  multiplicity drops, $F^{(1)}$ will be a {\GWT} equation.
\end{proof}

The argument also applies verbatim to the following case:

\begin{corollary}
Assume $\Gamma[k] (F)$ is a quadrant and let $\cS^{(1)}$ be a quadratic transformation of $\cS$. If $F^{(1)}$ is a local {\WT} equation of $\cS^{(1)}$, then $\Gamma[k] \left( F^{(1)} \right)$ is also a quadrant.
\end{corollary}

As a side and easy remark, as combinatorial monoidal transformations amount to translations, it is clear that these blow-ups also take quadrants into quadrants. 
So, quadrants are stable by the type of blow-ups we are using. 
		In this case, then, we need only to apply monoidal transformations in $(Z,X)$ and $(Z,Y)$ to resolve~$F$, since monoidal transformations amount to translations of $\Delta(F)$ (see Section~\ref{subsec:monoidal}).

A second advantage, if we are in the quadrant situation, is that the number of blow-ups that the surface can undergo before dropping in multiplicity is easily bounded (see Proposition \ref{cuadrante}).

\bigskip

So, how hard is to get to the quadrant case? It turns out that by Proposition~\ref{almostall} it is in fact very difficult {\em not to get} there. 
However the next example shows that in general we do not always get to a quadrant.

\begin{example}
Consider the surface defined by the equation
$$F=Z^n-X^r\biggl(Y-\sum_{q\geq 1}\alpha_qX^q\biggr)^s$$
with $r+s>n$. It is clear that performing quadratic transformations in the directions $(1:\alpha_q:0)$ we will never get to a quadrant.

\end{example}

We will prove in next section that 
this is essentially the only {\em bad case}. In fact, it was one of the reasons to introduce the notion of generalized quadrants (see Defintion \ref{def11} below). Moreover, as we will see in forthcoming sections we will apply transvections to get a {\GWT} equation, and obviously quadrants are not stable under such transformations, while generalized quadrants obviously are.

\section{Generalized Quadrants and Where to Find Them}

For this section we will assume $F$ to be a {\GWT} equation with $\overline{F}=Z^n$. 
We will define a more general notion, the generalized quadrants.

As we pointed out at the beginning of Section \ref{secquad}, when we study the effect of a blow-up (or a chain of blow-ups) in the Hironaka polygon, we are assuming that the multiplicity does not drop.

\begin{definition}
	A polygon $\Delta(F)$ (resp.~$\Gamma[k](F)$), is called a {\em generalized quadrant} if there exists a transvection $\varphi$ such that $\Delta\bigl(\varphi(F)\bigr)$ (resp.~$\Gamma[k]\bigl(\varphi(F)\bigr)$) is a quadrant.
 \label{def11}
\end{definition}

This will be particularly useful when looking at the polygons $\Gamma[k]$ and their evolution during the resolution process.

\begin{remark}
%	\quad
%	\begin{enumerate}
%		\item 
		Note that generalized quadrants appear naturally in our resolution process: if $\Delta(F)$ is a quadrant, and $F$~is {\WT}, but not {\GWT}, the transvection that takes $F$~into a {\GWT} equation also takes  $\Delta(F)$~into a generalized quadrant. 
%		
%		\item 
%	\end{enumerate}
\label{RemGWT}
\end{remark}

As proved in the previous section for quadrants, we can show that generalized quadrants are stable.

\begin{lemma}
Assume $\Gamma[k] (F)$ is a generalized quadrant and let $\cS^{(1)}$ be a quadratic transform of $\cS$. If $F^{(1)}$ is a local {\WT} equation of $\cS^{(1)}$ then $\Gamma[k] \left( F^{(1)} \right)$ is also a generalized quadrant.
\end{lemma}

\begin{proof}
Let $\varphi$ be the transvection that takes $\Gamma[k](F)$ into a quadrant, and let $b_k(X,Y)$ be the image of $a_k(X,Y)$ by this transvection. If we look in the direction $(1:\alpha:0)$, by Lemma \ref{transvectionsquare} there exist $\beta$ and $\psi$ such that $\psi\circ\varpi_{(1:\alpha:0)}=\varpi_{(1:\beta:0)}\circ\varphi$. By Proposition \ref{PSQ} $b_k(X,Y)$ is transformed into a quadrant, and hence the claim follows.

The claim is also clear for $\varpi_{(0:1:0)}$. Indeeed, if $\Gamma[k](F)$ is a generalized quadrant then 
$$
a_k(X,Y) = X^i \left( Y-\sum\alpha_lX^l \right)^j u_k(X,Y)
$$
for certain positive integers $i,j$ and unit $u_k$. The image of this coefficient under $\varpi_{(0:1:0)}$ is 
$$
X^iY^{i+j+k-n}\widetilde{u}_k(X,Y),
$$
which is a quadrant, and if we have lost the {\GWT} condition, it will turn into a generalized quadrant.
\end{proof}

 We will now focus  on a coefficient $a_k(X,Y)$ and its associated polygon $\Gamma[k]$. Recall from Section~\ref{S2} that:
\begin{itemize} 
\item The rightmost vertex of $\Gamma[k]$ is $\cR[k]=(\nu_k,0)$.
\item If $\Gamma[k]$ is not a quadrant, the slope $l$ of the compact edge containing $\cR[k]$ is, at most, $-1$. %that is:
%$$
%R[k]=(\nu_{k},0),\qquad 
%\text{where $\nu_k = \min \bigl\{  i+j\ %\bigm|\ (i,j) \in \Gamma[k]  \bigr\}$.}
%$$
\end{itemize}

The two situations that can appear, $l=-1$ and~$l<-1$, feature different (eventual) behaviors and will often be treated separately.

Our goal is to show that, in fact, $\Gamma[k]$ evolves to a generalized quadrant. We study the evolution of the polygon $\Gamma[k]$ under blow-ups in more generality than in section above but with the {\GWT} condition, and paying special attention to the point $\mathcal L[k]$.

As we have seen above, the monoidal transformations centered in $(Z,X)$, which are the only ones we are considering, do not affect the shape of the polygon $\Gamma[k]$, as they are merely translations (in their combinatorial version), so we should only be concerned about the quadratic transformations.

Let us consider $L[k]_2$, the $y$-coordinate of $\cL[k]$. As we noted in Remark~\ref{remT}, if we have a {\WT} equation and we want to take it into a {\GWT} form, the transvection(s) needed do not increase $L[k]_2$. Clearly, a descent of $L[k]_2$ is nice for our purposes, as $L[k]_2 \in \Z_{\geq0}$ and $L[k]_2 = 0$ if and only if $\Gamma[k]$ is a {\GWT} quadrant.

\medskip

Next we consider all possible transformations and argue that if $L[k]_2$ does not tend to zero it is because $\Gamma[k]$ is already a generalized quadrant.

\subsection{Combinatorial Transformations}
\label{subsec:5.1}

By Lemma \ref{2direcciones} the polygon $\Delta(F^{(1)})$ of the transform $F^{(1)}$ after $\varpi_{(1:0:0)}$ is a {\GWT} quadrant whenever $F$ is a {\GWT} equation. 
 We can apply as well this result to the polygon $\Gamma[k](F)$.

Now let us tackle the case of direction $(0:1:0)$.
%, where we have the following combinatorial transformation:
%$$
%(i,j)  \longmapsto  (i, i+j+k-n).
%$$
%This case is markedly different from the one above.
To begin with, it has the problem that the resulting equation might {\em not} be a {\GWT} equation. However, what is clear is that $L_2$ drops because $(Z,X)$ is not permissible.
%$\Gamma[k]^{(1)}$ must again be a quadrant, as
%$$
%(\nu_k,0)  \longmapsto  (\nu_k, %\nu_k+k-n) \in \Gamma[k]^{(1)},
%$$
%and all points on it have a greater or %equal $y$-coordinate. In fact 
%$$
%\Gamma[k]^{(1)} = (\nu_k, \nu_k+k-n) + %\R^2_{\geq0}  \iff  l<-1.
%$$
To come back to a {\GWT} equation, we need to apply a transvection.

This phenomenon, where the transform of a quadrant is a quadrant, but only after a change of coordinates, is the main idea that led us to the definition and use of generalized quadrants.

\subsection{The Direction $(1:\alpha:0)$: Easy Cases}

Moving on to the general case, and taking into account that the combinatorial transformation amounts to:
$$
(i,j)  \longmapsto   \bigl\{  (i+j+k-n,0),\dots, (i+j+k-n, j)  \bigr\},
$$
it is easy to see that, in the case $l<-1$, once again we have
$$
\Gamma[k]^{(1)} = (\nu_k+k-n,0) + \R^2_{\geq0}.
$$

Assume therefore $l=-1$, in which case we write
$$
\overline{a_k}(X,Y) = X^{r_0} (Y -\beta_1X)^{r_1}\dots (Y -\beta_sX)^{r_s},
$$
with $\nu_k = \sum_{i=0}^s r_i$, and let us denote
$$
\cM = \biggl( r_0, \sum_{j=1}^s r_i \biggr)
$$
the leftmost vertex of the compact facet containing $\cR$.

Then, after the transformation we get
$$
\overline{a_k}(X,Y)  \longmapsto  X^{\sum_{j=0}^s r_j - n} (Y + \alpha - \beta_1)^{r_1}\dots (Y + \alpha - \beta_s)^{r_s},
$$
which implies:
\begin{itemize}
\item If $\alpha \neq \beta_i$ for all $i=1,\dots,s$, then
$$
\Gamma[k]^{(1)} = \Biggl(\sum_{j=0}^s r_j - n , 0 \Biggr) + \R^2_{\geq0},
$$
and we are done.
\item If $s\geq 2$ and $\alpha = \beta_1$ (without loss of generality), then $\overline{a_k}$ is transformed into a polynomial of the form $X^{\nu_k-n}Y^{r_1}u(Y)$, with $u(Y) \in K[[Y]]^*$. Clearly then,
$$
\cL[k]^{(1)} = \left( \nu_k-n,r_1 \right),
$$
and, although the equation might not be {\GWT}, the resulting {\GWT} form will have the same $\cL[k]^{(1)}$ and
$$
\cR[k]^{(1)} = \left( \nu_k+r_1-n,0 \right).
$$
In any case, note that 
$$
L[k]_2 = \sum_{j=1}^s r_i > r_1 \geq L[k]_2^{(1)}.
$$
\end{itemize}

The missing case is, therefore, the one given by $\overline{a_k}(X,Y) = X^r (Y - \alpha X)^s$, and we will devote more time to it.

\subsection{The Direction $(1:\alpha:0)$: The Difficult Case}
\label{difficult}

We look now at the case where we have
$$
L[k]_2^{(1)} = L[k]_2,
$$
which may look as a no-progress situation towards our aim of taking $\Gamma[k]$ to a quadrant. To make it even worse, this situation can be extended to any number of quadratic transformations. Think of the following situation:
$$
a_k(X,Y) = X^r \left( Y - \alpha_1 X - \alpha_2 X^2 - \dots - \alpha_q X^q \right)^s,
$$
where $L[k]_2 = s$. Then, if we consider successive quadratic transformations in the directions $(1:\alpha_1:0)$, \dots,~$(1:\alpha_q:0)$ we have
$$
L[k]_2^{(q)} = \dots = L[k]_2^{(1)} = L[k]_2=s.
$$
Our  aim is to show that this can only happen essentially for such a coefficient $a_k(X,Y)$. That is, if a sequence of quadratic transformations does not decrease $L[k]_2$, then it is due to the fact that $a_k(X,Y)$ is the product of two curves: the exceptional divisor and another one, which determines in which direction one should look in order {\em not to} get any better. 

To prove this, first note that, by Lemma 2, successive quadratic transformations in the directions $(1:\alpha_1:0)$, \dots,~$(1:\alpha_q:0)$ are, up to an initial transvection, $q$ quadratic transformations in the direction $(1:0:0)$. 

So, let us suppose that $L[k]_2$ is invariant by $q$ quadratic transformations in the direction $(1:0:0)$, and write
$$
a_k(X,Y) = X^r b_k(X,Y) = X^r ( Y^s + \sum_{i>0}a_{i+rjk}X^iY^j ).
$$

Applying Weierstrass Preparation Theorem to $b_k(X,Y)$ we know that there exists $u_k(X,Y) \in K[[X,Y]]^*$ and
$$
c_k (X,Y) = Y^s + \sum_{j=0}^{s-1} c_{jk}(X) Y^j = Y^s + \sum_{i=0}^\infty \sum_{j=0}^{s-1} c_{ijk} X^iY^j,
$$
such that $b_k(X,Y) = u_k(X,Y) \cdot c_k(X,Y)$. Note that it must hold 
$$
\nu \bigl( c_{jk} (X) \bigr) \geq s-j,
$$ 
that is $i+j \geq s$, whenever $c_{ijk} \neq 0$.

We are assuming $L[k]_2^{(1)} = L[k]_2$ after a quadratic transformation on the direction $(1:0:0)$, but
\[
    \begin{aligned}
a^{(1)}_k &= X^{r+s+k-n} u_k(X,XY) \cdot \Biggl(  Y^s + \sum_{j=0}^{s-1} c^{(1)}_{k,j}(X) Y^j \Biggr) \\
&= X^{r+s+k-n} u_k(X,XY) \cdot \Biggl(  Y^s + \sum_{i=0}^\infty \sum_{j=0}^{s-1} c_{ijk} X^{i+j-s}Y^j \Biggr).
\end{aligned}
\]

Write $\In_X$ for the initial form with respect to the variable $X$. Note that 
$$
\In_X \bigl( a^{(1)}_k(X,Y) \bigr) = X^{r+s+k-n} Y^s  \iff  i+j > s, \mbox{ for all } c_{ijk} \neq 0,
$$
and both conditions are also equivalent to $L[k]_2^{(1)} = L[k]_2$. So, according to our hypothesis, both conditions hold. If we perform another quadratic transformation, we get
\[
    \begin{aligned}
a^{(2)}_k &= X^{r+2s+2k-2n} u_k(X,X^2Y) \cdot \Biggl(  Y^s + \sum_{j=0}^{s-1} c^{(2)}_{k,j}(X) Y^j \Biggr) \\
&= X^{r+2s+2k-2n} u_k(X,X^2Y) \cdot \Biggl(  Y^s + \sum_{i=0}^\infty \sum_{j=0}^{s-1} c_{ijk} X^{i+2j-2s}Y^j \Biggr),
\end{aligned}
\]
and once again,
\[
    \begin{aligned}
L[k]_2^{(2)} = L[k]_2 & \iff i+2j > 2s, \mbox{ for all } c_{ijk} \neq 0 \\ 
& \iff  \In_X \bigl( a^{(2)}_k(X,Y) \bigr) = X^{r+2s+2k-2n} Y^s.
\end{aligned}
\]

Therefore, the existence of $q$ successive quadratic transformations in the direction $(1:0:0)$ which leave $L[k]_2$ invariant is equivalent to the fact that, in the decomposition
$$
a_k(X,Y) = X^r \cdot u_k(X,Y) \cdot \Biggl( Y^s + \sum_{i=0}^\infty \sum_{j=0}^{s-1} c_{ijk} X^iY^j \Biggr),
$$
we have $i+qj>qs$, for all $(i,j)$ such that $c_{ijk} \neq 0$. 

So, in the Newton diagram of the polynomial $c_k(X,Y) \in K[[X]][Y]$, the invariance of $L[k]_2$ implies that there are no points in the regions shown in Figure~\ref{fig:9}, or, more precisely, each step leaving $L[k]_2$ invariant implies that the corresponding region has no points.

\begin{figure}[htbp]
    \begin{tikzpicture}[scale=1]
    \tikzset{
        every point/.style = {circle, inner sep={2\pgflinewidth}, 
            opacity=1, draw, solid, fill
        },
        point/.style={insert path={node[every point, #1]{}}}, point/.default={},
        point name/.style = {insert path={node[circle,inner sep={2\pgflinewidth}] (#1){}}},
        aux/.style = {black!20}
    }
 
    % \node[above] at (.5,4) {$\vdots$};
    % \fill[aux] (.5,4) -- (1,3) -- (3,0) -- (4,0) -- (4,4) -- cycle;
    % \draw[pattern=dots,] (1,3) -- (2,0) -- (3,0) -- cycle;
    % \draw[very thick] (.5,4) -- (1,3) -- (3,0);
    % \draw [] (1,3)--(2,0);
    \coordinate (O) at (0,0);
    \coordinate (L1) at (1.5,2);
    \coordinate (L) at (3,2);
    \coordinate (b1) at (4,0);
    \coordinate (b2) at (5,0);
    
    \draw[dashed] (L1) -- (L1|-O);
    \draw[dashed] (L) -- (L|-O);
    % \draw[pattern=dots] (L1) -- (b1) -- (b2) --cycle;
    \draw[fill=black!30] (L1) -- (b1) -- (b2) --cycle;
    \draw[] (L1) -- (b1) (L1)--(b2);
    \draw (L) node[point] {} (L1) node [point] {};
    \node[left] at (L1) {$\mathcal{L}[k]^{(q-1)}$};
    \node[right] at (L) {$\mathcal{L}[k]^{(q)}$};
    
    \node[yshift=-.5cm, xshift=-.25cm] (slope) at  (L|-O) {slope $-1/q$};
    \draw [->, shorten >=2pt] (slope.north) to[out=90, in=45+180] ($(L1)!.5!(b1)$) ;
    \node[yshift=1.5cm] (slope1) at (b2) {slope $-1/(q+1)$};
    \draw[->, shorten >=2pt] (slope1.south) to[out=-90, in=45] ($(L1)!.80!(b2)$);

    \draw[->] (0,0) -- (6,0);
    \draw[->] (0,0) -- (0,3);
    \end{tikzpicture}
    \caption{If $L[k]_2^{(q-1)}=L[k]_2^{(q)}$ there cannot be any point of $\Gamma[k]$ in the shaded region.}
    \label{fig:9}
\end{figure}
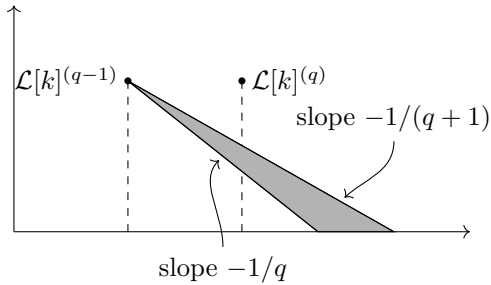

Therefore we have the following result.

\begin{proposition}
Let $F$ be a {\GWT} equation and $a_k(X,Y)$ one of its coefficients. If  we have $q$ successive quadratic transformations in the directions $(1:\alpha_1:0),\dots,(1:\alpha_q:0)$, with the resulting equations verifying
$$
L[k]_2^{(q)} = \dots = L[k]_2^{(1)} = L[k]_2=s.
$$

Then $a_k(X,Y)$ can be written in the following form:
$$
a_k(X,Y) = X^r \cdot u_k(X,Y) \cdot \Biggl[  \Biggl( Y - \sum_{i=1}^q \alpha_i X^i \Biggr)^s + \sum_{i>q(s-j)} \sum_{j=0}^{s-1} c_{ijk} X^iY^j \Biggr].
$$
\label{monster}
\end{proposition}

Note that the last summand consists of monomials having large degrees (so large that they actually do not interfere substantially in the evolution of the equation). In particular, if there is a monomial in one of the regions of Figure~\ref{fig:9} (say $x+qy \leq qs$), one {\em cannot} have $q$ successive quadratic transformations which leave $L[k]_2$ invariant.

Therefore, if $a_k$ verifies that after successive quadratic transformations in certain directions $\bigl\{(1:\alpha_i:0)\bigr\}$, $L[k]_2$ never drops, we have by the previous proposition that 
\begin{equation}
a_k(X,Y)=X^r \Biggl( Y-\sum_{q\geq 1}\alpha_qX^q \Biggr)^s u_k(X,Y),
\label{monstruo}
\end{equation}
or in other words, $a_k(X,Y)$ is~a generalized quadrant (and so are all $a_k^{(i)}(X,Y))$.

Summing up, in a finite number of steps we get to a generalized quadrant. Hence

\begin{proposition}
Let 
\[
    F=Z^n+\sum_{k=0}^{n-2} a_k(X,Y)Z^k
\]
be a {\WT} equation with $\overline{F}=Z^n$. Following Algorithm~\ref{alg}, in a finite number of steps 
% every coefficient $a_k(X,Y)$ is transformed into a generalized quadrant.
we arrive at a surface with a defining equation $G$ such that every $\Gamma[k](G)$ is a generalized quadrant.
\label{RemS}
\end{proposition}

Unfortunately it is not possible to bound the number of blow-ups we need to perform to get to the situation described in Proposition~\ref{RemS}, neither in terms of the polygon $\Delta(F)$ nor of the polygon $\Gamma[k]$, as the next example shows. 

\begin{example}
    Consider
    \[a_k(X,Y)=X^r\biggl(Y-\sum_{q\geq 1}\alpha_q X^q\biggr)^su_k(X,Y)+X^mY^l\]
 with $u_k$ a unit, $m>r+s$ and $l\in\Z_{l\geq 0}$. The number of blow-ups we need to perform before we get to a generalized quadrant depends on $m$, but $m$ cannot be read off neither $\Delta(F)$ nor $\Gamma[k]$.
\end{example}

The situation where every coefficient gives a generalized quadrant is much easier to control. We will not need in fact to use {\GWT} equations from this point on. In particular, we obtain the following result.

\begin{proposition}
Let 
\[
    F=Z^n+\sum_{k=0}^{n-2} a_k(X,Y)Z^k
\]
be a {\WT} equation. Then there exists an integer $r \geq0$ verifying that, up to some transvection, every $\Gamma[k]\left( F^{(r)} \right)$ is a quadrant.
\label{PF}
\end{proposition}

\begin{proof}
By Propostion \ref{RemS}, in a finite number of steps we arrive at the situation where every $\Gamma[k](F)$ is a generalized quadrant. We will prove that by blowing up further we get to the situation of the claim. As always, we only have to be concerned about quadraic transformations, since monoidal transformations correspond to translations of the polygons.

If we choose the direction $(0:1:0)$ at any moment, generalized quadrants turn into quadrants, so we are done. Now, for every $k=0,...,n-2$ there is a transvection $\varphi_k$ such that $\Gamma[k](\varphi_k(F))$ is a quadrant. Consider the numbers:
$$
\mu_{ij} = \left\{ \begin{array}{ll} 
\nu \left( \varphi_i - \varphi_j \right) & \mbox{ if } \varphi_i \neq \varphi_j \\
0 & \mbox{ otherwise }
\end{array} \right.
$$
and let us define 
$$
r = \max_{i<j} \left\{ \mu_{ij} \right\}.
$$

It is clear that, after $r$ transformations (assuming the multiplicity has not changed) with directions $\{(1:\alpha_1:0),...,(1:\alpha_r:0)\}$, the coefficients of $F^{(r)}$ fall in two sets:
\begin{itemize}
\item Those verifying
$$
\varphi_k (Y) = Y + \beta_1 X + ... + \beta_r X^r + ..., 
$$
where, for some $j \in \{1,...,r\}$, $\alpha_j \neq \beta_j$. Then, precisely in the $j$--th quadratic transform $\Gamma[k] \left(F^{(j)}\right)$ is a {\GWT} quadrant. And of course it stays a quadrant for the transformations to come.
\item Those verifying 
$$
\varphi_l (Y) = Y + \alpha_1 X + ... + \alpha_r X^r + X^r \psi_l(X), \mbox{ with } \nu \left( \psi_l (X) \right) > 1.
$$
By definition of $r$, all these coefficient must have the same associated transvection.
\end{itemize}

Then, if we consider $F^{(r)}$, the transvection (in the notation above)
$$
\varphi(Y) = Y + \psi_l(X),
$$
which might well be the identity, takes into a quadrant all the $\Gamma[l]\left( F^{(r)} \right)$ which have not being taken into {\GWT} quadrants previously. Note that those {\GWT} quadrants are invariant under transvections
\end{proof}

\begin{remark}
The process of passing from the condition all $\Gamma[k]$ generalized quadrants to all $\Gamma[k]$ quadrants in the previous result is effectively computable (following the proof), not in terms of the combinatorics, but in terms of the factorization of the coefficients.
\end{remark}

Once we have proved these results we can answer our original question: does $\Delta(F)$ tend to a quadrant? The answer is in the affirmative although, of course, before the quadrant situation is reached it may happen that the multiplicity drops or the tangent cone ceases to be a plane (which will receive its own formal treatment in the next section).

\begin{proposition}
Let $\cS$ be an embedded algebroid surface of multiplicity $n>1$, $F$ an equation with $\overline{F}=Z^{n}$. Applying Algorithm~\ref{alg}, there exists an integer $r\geq 0$ such that $\overline{F^{(r)}}\neq Z^{n}$ or $\Delta(F^{(r)})$ is a quadrant.
\end{proposition}

\begin{proof}
We can assume now that $F$ verifies that $\Gamma[k](F)$ is a quadrant for $k=0,...,n-2$. Let us write
$$
\Gamma[k](F) = \left( a_k, b_k \right) + \Z^2_{\geq 0}.
$$

As monoidal transformations take quadrants into quadrants we only have to be concerned about quadratic transformations. Now, if $\alpha \neq 0$, a quadratic transformation in the direction $(1:\alpha:0)$ takes 
$\Gamma[k](F)$ into
$$
\Gamma[k] \left( F^{(1)} \right) = \left( a_k + b_k - k, 0 \right) + \Z^2_{\geq 0},
$$
therefore it is a {\GWT} quadrant (and so is $\Delta \left( F^{(1)} \right)$ henceforth).

On the other hand, if we perform successive combinatorial quadratic transformations, the Corollary 1 tell us that $d(\cL, \cR)$ decreases in each transformation, hence we arrive to a quadrant situation eventually.
\end{proof}

This is, therefore, our solution to Problem 6 in \cite{HS}, for the algebroid situation.

Of course it may happen that the multiplicity drops before we get to the quadrant situation, as the following example shows.

\begin{example}
Consider the surface defined by 
$$F=Z^n+X^{n+1}+Y^{n+1}$$
Blowing up the origin once and the multiplicity drops in any direction of the tangent cone.
\label{ejCP}
\end{example}

\section{Bounding the Resolution Process}
\label{Sec6}
Our aim in this section will be to prove the following result.

\begin{theorem}
Let $\cS$ be an embedded algebroid surface defined by a {\GWT} equation $F\in K[[X,Y]][Z]$. Assume we apply the following algorithm:
\begin{itemize}
\item If $(Z,X)$ or $(Z,Y)$ are permissible curves, we make a monoidal transformation centered in them.
\item Otherwise we perform a quadratic transformation.
\end{itemize}
Then the multiplicity of $\cS$ drops in a finite number of steps.
\label{Teorema}
\end{theorem}

Note that, what differs from Levi-Zariski strategy is that the only permissible curves we blow-up are $(Z,X)$ and $(Z,Y)$. \emph{At the same time}, and starting from a {\GWT} equation, we will try to bound the number of blow-ups we perform before the multiplicity drops. 

The study of the evolution of the combinatorial objects through the resolution process will be split in two parts depending on $\overline{F}$, as the arguments involved are rather different, both algebraically and geometrically.

\subsection{Bounding the Resolution Process (I): The Tangent Cone is not a~Plane}
\label{sec:I}

Let $\cS$ be an embedded algebroid surface and $F$ a {\GWT} equation as above. We will study the case where $\overline F$ is not the power of a linear form (equivalently, $\cC$ is not a plane). In this case, therefore, we must have $\cR=(1,0)$.

Assume that we can perform a combinatorial monoidal transformation. There cannot be more than one permissible curve and it must be $(Z,X)$. Then, $F$~being a {\GWT} equation, $\Delta(F)$ can only be a quadrant of vertex $(1,0)$ (otherwise $(Z,X)$ would not be permissible). In particular $\overline F$ does not depend on $Y$,
$$
\overline F=Z^n+\sum_{i+k=n}a_{i0k}X^iZ^k=\prod_{l=1}^n(Z-\alpha_l X), \mbox{ with all } \alpha_l \neq 0.
$$

We have to look now at the equation for the monoidal transformation in a generic direction $(1:0:\alpha_{l_0})$ of the exceptional divisor. This equation is
$$
F^{(1)}=(Z+\alpha_{l_0})^n+\sum a_{ijk}X^{i+k-n}Y^j(Z+\alpha_{l_0})^k.
$$

Here we find, from $\overline F$, some monomials corresponding to the product
$$
\prod_{l=1}^n(Z+\alpha_{l_0}-\alpha_l) = \overline{F} ( 1,Z + \alpha_{l_0} )
$$
but not all of the monomials, different from $Z^n$, can cancel. For instance, the coefficient of $Z^{n-1}$ is
$$
\sum_{l=1}^n ( \alpha_{l_0} - \alpha_l ) = n \alpha_{l_0} \neq 0,
$$
taking into account that $\sum \alpha_l = 0$, since in particular $F$ is a {\WT} equation and therefore $a_{n-1}(X,Y)=0$.

Hence, there is a monomial in $F^{(1)}$ of order strictly smaller than $n$ and we are done. Therefore, we can suppose that $(Z,X)$ is not a permissible curve and we are forced to apply a quadratic transformation. 

Let us choose a point $(\alpha:\beta:\gamma) \in \cC$ and assume first that $\alpha=1$, without loss of generality. Then the equations of the blow-up are:
$$
X \longmapsto X, \quad Y \longmapsto X ( Y + \beta), \quad Z \longmapsto X ( Z + \gamma),
$$
and the equation $F^{(1)}$ looks like
$$
F^{(1)} = (Z + \gamma)^n + \text{(terms of degree $\leq (n-2)$ in $Z$)}.
$$

Therefore, if $\gamma \neq 0$ there is at least a monomial of degree $n-1$ and we are done. Let us have a look at the case of direction $(1:\beta:0)$. Take some $a_k$ such that $\nu(a_k) = n-k$, which must always exist in this case.

Then we must have
$$
\overline{a_k}(X,Y)Z^k = \sum_{i+j=n-k} a_{ijk} X^iY^jZ^k = X^r (\beta_1 X - Y)^{s_1} \dots (\beta_l X - Y)^{s_l}Z^k,
$$
with $r \neq 0$ and $r + \sum s_i = n-k$. These monomials are transformed by the quadratic transformation in the direction $(1:\beta:0)$ into
$$
\sum_{i+j=n-k} a_{ijk} ( Y + \beta )^j Z^k = (\beta_1 - \beta + Y)^{s_1} \dots (\beta_l - \beta + Y)^{s_l}Z^k,
$$
and for similar reasons as above, we have a monomial of degree strictly smaller than~$n$.

The only remaining case is that of direction $(0:1:0)$. As we are assuming now that $(Z,X)$ is not permissible, there must be some $(\lambda, \mu) \in \Delta(F)$ such that $\lambda <1$. In fact, $\cL$ must be in this situation. The effect of this transformation in combinatorial terms amounts to:
$$
(\lambda, \mu) \longmapsto (\lambda, \lambda + \mu - 1).
$$

\begin{figure}[htbp]
    \begin{tikzpicture}[scale=1]
    \tikzset{
    every point/.style = {circle, inner sep={2\pgflinewidth}, 
        opacity=1, draw, solid, fill
    },
    point/.style={insert path={node[every point, #1]{}}}, point/.default={},
    point name/.style = {insert path={node[circle,inner sep={2\pgflinewidth}] (#1){}}},
    aux/.style = {black!20}
    }

    \coordinate (O) at (0,0);
    \coordinate (B) at (4,4);
    \coordinate (A1) at (.5,3.5);
    \coordinate (A2) at (.75,2.5);
    \coordinate (A3) at (2,.75);
    \coordinate (A4) at (2.75,0);
    
    \fill [black!20] (B-|A1) -- (A1) -- (A2) --(A3) --(A4) --(A4-|B) --(B) --cycle;
    \draw[->] (O)  -- (O|-B)  ;
    \draw[->] (O)  -- (O-|B)  ;
    \draw[] (B-|A1)--(A1)--(A2) (A3)--(A4)--(A4-|B);
    \draw[dashed] (A2)-- (A3);
    
    \coordinate (L) at (A4-|B);
    
    \draw (A1)--(A4);
    \draw (A1) [point];
    \draw (A4) [point];
    \node [below] at (A4) {$\mathcal{R}$};
    \node [left]  at (A1) {$\mathcal{L}$};
    \pic [draw, ->, "$\theta$", angle eccentricity=1.5] {angle = L--A4--A1};
    \end{tikzpicture}
    \caption{The angle $\theta$ measures the total number of blow ups when the tangent cone is not a plane.}
\end{figure}
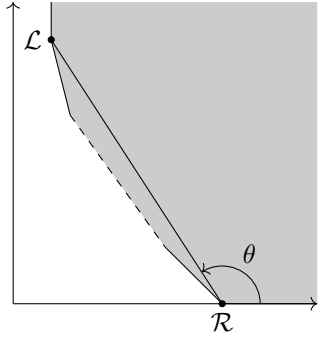

So, let us call $\theta$ the angle defined by the point of the $OX$-axis at infinity, $\cR = (1,0)$ and $\cL = (\lambda, \mu)$. Only $\cL$ is changed by the transformation, and it changes as above. The change of $\theta$ (to, say $\theta ^{(1)}$) can be measured clearly on its tangent, as
$$
\tan (\theta) = \frac{\mu}{\lambda - 1} \longmapsto \tan \big( \theta^{(1)} \big) = \frac{\lambda + \mu - 1}{\lambda - 1} = \tan (\theta) + 1.
$$

Eventually, after $[-\tan(\theta)]$ iterations, we will have that the tangent of the resulting angle will be in $(-1,0]$. This implies there is a monomial with degree strictly smaller than $n$ and we are done. 

So we have proved the following result.

\begin{theorem}
If $\cC$ is not a plane, the maximum number of blow-ups that can be performed in $\cS$ before the multiplicity drops is:
\begin{itemize}
\item $1$, if $\cL = \cR$.
\item $[-\tan (\theta)]$ otherwise, with the previous notations.
\end{itemize}
\label{cota1}
\end{theorem}

In previous sections we studied the evolution of the Hironaka polygon under blow-ups, and proved that it tends to a quadrant in the case of plane tangent cone. Note that in the case of non-plane tangent cone we have just seen that the polygon $\Delta(F)$ tends to a quadrant too, but it mostly happens that the multiplicity drops before we get it.

% \section{Effective combinatorial surface resolution}
\subsection{Bounding the Resolution Process (II): The Tangent Cone is a Plane}
\label{sec:II}

As we have seen, combinatorics may be blind to the effect of blow-ups, which makes it impossible to bound the number of transformations we need to perform before multiplicity drops. What we can nevertheless do is to {\em prepare} the equation $F$ and bound the resolution process \emph{ab initio} from the prepared equation. 

After Proposition \ref{RemS} and Proposition \ref{PF} we have two ways to prepare the equation. In both cases we are able to bound the number of blow-ups needed before the multiplicity drops. But it seems to be non-trivial to know which condition is stronger (though we have used Proposition \ref{RemS} to prove Proposition \ref{PF}).

\begin{proposition}
Let $F$ be a {\WT} equation and assume $\Delta (F)$ is a quadrant, $\Delta (F) = \cL + \Q_{\geq0}^2$. Then the maximum number of blow-ups that can be performed before the multiplicity drops is $[L_1]+[L_2]+n$. 
\label{cuadrante}
\end{proposition}

\begin{proof}
By hypothesis $\mathcal L=\mathcal R=(L_1,L_2)$. We know
$$
(Z,X) \mbox{ is permissible }  \Longleftrightarrow  L_1 \geq 1,
$$
and, should this be the case, we have $L_1^{(1)} = L_1 -1$, as the monoidal transformation amounts to a translation. Clearly we can repeat this process no more (and no less) than $[L_1]$ times.

The same goes with $(Z,Y)$ and $L_2$, and therefore, after $t=[L_1] + [L_2]$ monoidal transformations we have an equation $F^{(t)}$ such that
$$
\Delta \bigl( F^{(t)} \bigr) = \bigl( L_1-[L_1], L_2-[L_2] \bigr) + \Q_{\geq0}^2.
$$

If $F^{(t)}$ still has multiplicity $n$, that implies:
\begin{itemize}
\item Neither $(Z,X)$ nor $(Z,Y)$ are permissible curves.
\item $1 \leq L_1+L_2-t < 2$.
\end{itemize}
We therefore perform a quadratic transformation which gives us a new equation $F^{(t+1)}$ such that
$$
\Delta \bigl( F^{(t+1)} \bigr) = \left( a,b \right) + \Q_{\geq0}^2.
$$

Now consider the following cases:
\begin{enumerate}
\item[(a)] If the direction considered is $(1:\alpha:0)$, with $\alpha \neq 0$, then
$$
(a,b) = ( L_1+L_2-t-1, 0 ),
$$
and therefore the multiplicity has dropped. 
\item[(b)] If the direction considered is $(1:0:0)$, then
$$
(a,b) = \bigl( L_1+L_2-t-1, L_2-[L_2] \bigr),
$$
and therefore the quadrant moves leftwards $1-\bigl(L_2-[L_2]\bigr) \geq 1/n$. Hence the bound follows. 

The case $(0:1:0)$ is analogous.
\item[(c)] If the considered direction is $(\alpha:\beta:\gamma)$ with $\gamma\neq 0$, the tangent cone cannot be a plane. Then, as we saw in the previous section ({\WT} is enough for this), the multiplicity drops unless
%, at most, if 
the 
direction considered is $(0:1:0)$, which is already taken care of in case~(b).
\end{enumerate}
This finishes the proof.
\end{proof}

\begin{corollary}
Let $F$ be a {\GWT} equation and assume $\Delta (F)$ is a {\GWT} quadrant, $\Delta (F) = (L_1,0) + \Q_{\geq0}^2$. Then $(Z,X)$ is permissible and, after $[L_1]$ monoidal transformations centered in $(Z,X)$ the multiplicity drops.
\label{Cor}
\end{corollary}

\begin{proof}
Clear from the properties of {\GWT} equations and the action of monoidal transformations on $\Delta(F)$.
\end{proof}

The condition of $\Delta(F)$ being a quadrant is quite strong. However we present here an example of a well-known family of surface singularities with the property of $\Delta(F)$ being a quadrant. 

    Let us consider the ring $S=K((X^{1/n}, Y^{1/n}))$ and an element $\zeta\in S$ of the form
\[
    \zeta=X^{a/n}Y^{b/n}u(X^{1/n},Y^{1/n}), \qquad u(0,0)\neq 0.
\]
By definition, a $\nu$-quasi ordinary series~$F$ is the minimum polynomial of~$\zeta$ over~$R=K((X,Y))$. We may identify the rings $S$ and~$K((X,Y))(X^{1/n}, Y^{1/n})$, and it is well known that $\operatorname{Gal} ( S/R)\simeq C_{n}\times C_{n}$. If  $\delta$ is a primitive $n$-th  root of unity, the automorphism corresponding to the pair $(i,j)$ is simply
\[
    \sigma_{(i,j)} \colon
    \left\{
        \begin{aligned}
            X^{1/n} &\mapsto \delta^{i} X^{1/n},\\
            Y^{1/n} &\mapsto \delta^{j} Y^{1/n}.
        \end{aligned}
    \right.
\]
We denote
\[
    \zeta_{(i,j)}=\sigma_{(i,j)}(\zeta)=\delta^{ai+bj}X^{a/n}Y^{a/n}u_{(i,j)}(X^{1/n}, Y^{1/n}).
\]
We know that the polynomial
\[
    F=\prod_{1\leq i<j\leq n} (Z-\zeta_{ij})=
    Z^{n}+\sum_{l=0}^{n-1} a_{n-l}(X,Y) Z^{l}
    \in K[[X,Y]][Z]
\]
is either the minimal polynomial of~$\zeta$ or a power of the minimal polynomial.
By requiring $F$ to be irreducible, we can suppose that $F$~is in fact the minimal polynomial of~$\zeta$. By general Galois Theory, we have actually~$F\in K[[X,Y]][Z]$.

Here, we have the Cardano relations
\[
    a_{n-l}(X,Y)=(-1)^{n-l}\sum \zeta_{(i_{1},j_{1})} \dots \zeta_{(i_{l}, j_{l})},
\]
where the $\zeta_{(i_{k}, j_{k})}$ are all different.

If we make the additional assumption that $u(0,0)=1$ we can write
\[
    a_{n-l}(X,Y)=(-1)^{n-l}\sum \delta^{a(i_{1}+\dots+i_{l})+b(j_{1}+\dots+j_{l})} X^{al/n} Y^{bl/n} v(X,Y),
\]
where $v(X^{1/n},Y^{1/n})$ is a unit with $v(0,0)=1$.

The independent term~$a_{n}$, being the product of all conjugates, will be of the form
$a_{n}=X^{a}Y^{b}w$, with $w$~a unit.
We know that in this situation, $F$ is a {\WT}-polynomial if and only in $\zeta$ has no integer exponents.

\begin{proposition}
    Let $F$ be a $\nu$-quasi ordinary series which is also a {\WT} polynomial. Then,
    \[
        \Delta(F)=(a,b)+\mathbb{R}_{\geq 0}^{2},
    \]
    in other words, the Newton-Hironaka polygon of~$F$ is a quadrant.
\end{proposition}

\begin{proof}
    Note that the hypothesis of being a {\WT}-polynomial is needed in order to assert that the Hironaka polygon is exactly the polygon given by~$F$.
    The rest of the proof is trivial: Note first that $(a,b)\in\Delta(F)$. It is immediate to check that projection and scaling of all possible exponents of $a_{n-l}$ do belong to~$(a,b)+\mathbb{R}_{\geq 0}^{2}$. Hence the fact.
\end{proof}

\bigskip

By Proposition \ref{PF} in a finite number of steps we get to the situation where every coeffcient $a_k(X,Y)$ is a  quadrant.

\begin{definition}
	We will say that a {\WT} equation $F$ is {\em prepared} if every $\Gamma[k](F)$ is a quadrant.
\label{prep}
\end{definition}

We can bound the number of blow-ups before a decrease in multiplicity for {\em prepared} equations.

\begin{theorem}
If $F$ is a {\em prepared} equation then following Algorithm~\ref{alg}
the multiplicity drops in less than $n(R_1+L_2-1)+1$ transformations.
\label{cuenta}
\end{theorem}

\begin{proof}
First note that the condition of $F$ being prepared is preseved during our resolution process.

Suppose that neither $(Z,X)$ nor $(Z,Y)$ are permissible for $F$. We study the effect of the quadratic transformation in any direction of the tangent cone.

 \begin{itemize}
\item In the direction  $(1:0:0)$ we have
\[
\begin{aligned}
    L_2^{(1)} &\leq L_2\\
    R_1^{(1)} &= R_1 + R_2 - 1 < R_1,
\end{aligned}
\]
where the last inequality follows because $(Z,Y)$ is not permissible, and hence $R_2<1$.

\item In the direction  $(0:1:0)$ we have, analogously,
\[
\begin{aligned}
    L_2^{(1)} &= L_1 + L_2 - 1 < L_2\\
    R_1^{(1)} &\leq R_1.
\end{aligned}
\]

\item In the direction  $(1:\alpha:0)$ with $\alpha\neq 0$, we consider the point $\mathcal M=(M_1,M_2)$ minimizing $P_1+P_2$ among the points $(P_1,P_2)$ in $\Delta(F)$ (in case there is more than one such point, we take the one with maximal height). Since $F$ is prepared, $\mathcal M$ corresponds to the vertex of a certain $\Gamma[k](F)$. In its transform there can be no cancellations and therefore, after the blow-up, in any direction $(1:\alpha:0)$ we will get that $\Delta(F^{(1)})$ is a quadrant of vertex $(M_1+M_2-1,0)$. Therefore
$$
L_2^{(1)} + R_1^{(1)} = M_1 + M_2-1 < R_1 + L_2.
$$
\end{itemize}
Hence, blowing up the origin in any direction of the tangent cone, we have proved that $L_2+R_1$ drops. And at least we have
$$
L_2^{(1)}+R_1^{(1)}\leq L_2+R_1-\frac{1}{n}
$$
If we perform a monoidal transformation the situation is even better, as we know that in this case
$$
L_2^{(1)} + R_1^{(1)} = L_2 + R_1 - 1.
$$

So, clearly after at most $n(R_1+L_2-1)$ steps we will be in a situation where the tangent cone is not a plane. Again, as all $\Gamma[k](F)$ will be quadrants, it is easy to see that a single (quadratic or monoidal) transformation will decrease the multiplicity. This finishes the proof.
\end{proof}

Notice that if $\Delta(F)$ is itself a quadrant and $F$ is prepared we can apply both Theorem \ref{cuenta} and Proposition \ref{cuadrante}. It is not possible in general to know which bound is better (see examples in Section \ref{ejemplos}).

\section{Examples}
\label{ejemplos}

The bounds from Theorem \ref{cota1} and Theorem \ref{cuenta} can be optimal, as the following examples show.

\begin{example}
Consider the surface defined by the {\GWT} equation $$F=Z^2+X^2+Y^{2r}.$$
The curves $(Z,X)$ and $(Z,Y)$ are not permissible, hence we blow-up the origin. In any direction different from $(0:1:0)$ the surfaced is resolved. While, after $r-1$ quadratic transformations in the direction $(0:1:0)$ we get
$$F^{(r-1)}=Z^2+X^2+Y^2$$
and after one more blow-up on any direction the surface is resolved, as there are no points of order $2$ at the tangent cone $\mathcal C$.

On the other hand, by Theorem \ref{cota1} we get the bound $r$, since $\mathcal L=(1,0)$ and $\mathcal R=(0,r)$.
\end{example}

\begin{example}
For a case where the tangent cone is a plane, consider the surface defined by the {\GWT} equation $$F=Z^n+X^{n-1}Y^{n-1},$$
which is already a prepared equation, since $\Delta(F)$ is a quadrant.

The curves $(Z,X)$ and $(Z,Y)$ are again not permissible, hence we blow-up the origin. While, as predicted in Proposition \ref{almostall}, most choices of directions resolve the singularity, if we perform a combinatorial quadratic transform, say in the direction $(0:1:0)$, we get
$$
F^{(1)} = Z^n + X^{n-1}Y^{n-2}
$$

It is plain to see that we can do this $n-2$ times, no permissible curves appear in the process, until we get
$$
F^{(n-2)} = Z^n + X^{n-1}Y,
$$
a surface with a non--planar tangent cone which is resolved by a single quadratic transform. Therefore the maximal number of blow--ups we can do without decreasing the multiplicity is precisely $n-1$, and then the bound from Theorem \ref{cuenta} is
$$
n \left( \frac{n-1}{n} + \frac{n-1}{n} - 1 \right) + 1 = n-1.
$$
\end{example}

Interestingly enough, the example above does {\em not} show that the bound from Proposition \ref{cuadrante} is optimal. We present next another example of $\Delta(F)$ quadrant where the best bound, though not optimal, is given by Proposition \ref{cuadrante}.

\begin{example}
Consider the surface defined by the {\GWT} equation
\[F=Z^n+X^{2n-1}Y^{2n-1}\]
If we blow up the permissible curves $(Z,X)$ and $(Z,Y)$ we get the equation of the previous example. Hence the surface is resolved in $n+1$ transformations. Now, check that by Proposition \ref{cuadrante} the bound is 
\[\left[\frac{2n-1}{n}\right]+\left[\frac{2n-1}{n}\right]+n=n+2\]
while by Theorem \ref{cuenta} we get the bound
\[n\left(\frac{2n-1}{n}+\frac{2n-1}{n}-1\right)+1=3n-1\]
which is worse because $n>1$.
\end{example}

As we see the bounds are far from being optimal in general; as one can guess from
our arguments, where extremely coarse bounds have to be taken in order to cover
pathological situations which may, in fact, never happen.

\begin{example}
Consider the surface defined by the {\WT} equation
\[F=Z^5+X^2YZ^3+X^3Y^3.\]
The tangent cone is the plane $Z=0$ and neither $(Z,X)$ nor $(Z,Y)$ are permissible. If we blow up the origin the muliplicity drops in the directions $(1:0:0)$ and $(0:1:0)$, while the transform is already smooth in the directions $(1:\alpha:0)$. Obviously our  equation is prepared with $\mathcal \mathcal L=(\frac{3}{5},\frac{3}{5})$ and $\mathcal R=(1,\frac{1}{2})$, and the bound given in Theorem \ref{cuenta} is $4$.
\end{example}

\section*{Acknowledgements}

The first author was supported by Project {\em Métodos Computacionales en Ál\-ge\-bra, D-módulos, Teoría de la Representación y Optimización (MTM2016-75024-P)} (Ministerio de Econom\'{\i}a y Competitividad). The second and third authors were supported by Project {\em Geometría  Aritmética, D-módulos y Singularidades  (MTM\-2016–75027–P)} (Ministerio de Econom\'{\i}a y Competitividad) and Project {\em Singularidades, Geometría Algebraica Aritmética y Teoría de Representaciones: Estructuras y Métodos Diferenciales, Cohomológicos, Combinatorios y Computacionales (P12–FQM–2696)} (Jun\-ta de Andaluc\'{\i}a and FEDER).


\begin{thebibliography}{99}

\bibitem{Abh}
{\sc S.S. Abhyankar,} {\em Resolution of singularities of embedded algebraic surfaces}. Academic Press, New York (1966).

\bibitem{CST} {\sc H. Cobo; M. J. Soto; José M. Tornero,}  `Blurred Combinatorics in Resolution of Singularities: (A Little) Beyond the Characteristic Polytope', To appear at {\em Kyoto Journal of Mathematics}.


\bibitem{CO}
{\sc V. Cossart,} {\em Resolution of Surface Singularities: Three Lectures}. Lecture Notes in Mathematics 1101. Springer (1984).

\bibitem{CP1}
{\sc V. Cossart; O. Piltant}, `Resolution of singularities of threefolds in positive characteristic. I. Reduction to local uniformization on Artin-Schreier and purely inseparable coverings'. {\em J.  Algebra} 320 (2008) 1051--1082.

\bibitem{CP2}
{\sc V. Cossart; O. Piltant,} `Resolution of singularities of threefolds in positive characteristic II', {\em J. Algebra} 321 (2009) 1836--1976.

\bibitem{KV}
{\sc K. Kiyek; J.L. Vicente}, {\em Resolution of curve and surface singularities. In characteristic zero}. Algebras and Applications 4. Kluwer (2004).

\bibitem{HH}
{\sc H. Hauser}, `Excellent surfaces and their taut resolution', in {\em Resolution of Singularities} (eds.: H. Hauser, J. Lipman, F. Oort, A. Quir\'os). Progress in Mathematics 181 341--374, Birkh\"auser (2000).

\bibitem{HS}
{\sc H. Hauser; J. Schicho}, `Forty questions on singularities of algebraic varieties'. {\em Asian J. Math.} 15 (2011) 417--436. 

\bibitem{HW}
{\sc H. Hauser; D. Wagner,} `Alternative invariants for the embedded resolution of purely inseparable surface singularities', {\em Enseign. Math.} 60 (2014) 177--224.

\bibitem{H1}
{\sc H. Hironaka}, `Characteristic polyhedra of singularities', {\em J. of Math. Kyoto Univ.} 7 (1967) 251--293.

\bibitem{H2}
{\sc H. Hironaka}, `Schemes, etc.', in {\em Algebraic Geometry, Oslo 1970}, Proc. of the 5th Nordic Summer School in Mathematics (ed.: F. Oort), 291--313, Wolters-Noordhoff Publising (1972).

\bibitem{Bowdoin}
{\sc H. Hironaka}, `Desingularization of excellent surfaces', in {\em Resolution of surface singularities} (eds.: V. Cossart, J. Giraud and M. Hermann), Lecture Notes in Mathematics 1101 99--312, Springer (1984).

\bibitem{J}
{\sc H.W.E. Jung} `Darstellung der Funktionen eines algebraischen K\"orpers zweier unabh\"angigen Ver\"anderlichen $x$, $y$ in der Umgebung $x=a$, $y=b$', {\em Journal f\"ur Reine und Angewandte Mathematik} 133 (2008) 289--314.

\bibitem{L2}
{\sc B. Levi}, `Risoluzione delle singolarit\`a puntualli delle superficie algebriche', {\em Atti Acad. Sci. Torino} 33 (1897) 66--86.

\bibitem{Sp2}
{\sc M. Spivakovsky}, `A counterexample to the theorem of Beppo Levi in three dimensions', {\em Invent. Math.} 96 (1989) 181--183.

\bibitem{PT}
{\sc R. Piedra; J.M. Tornero}, `Equimultiple locus of embedded algebroid surfaces and blowing-up in characteristic zero'. {\em Serdica Math. J.} 30 (2004) 195--206.

\bibitem{LZ0}
{\sc R. Piedra; J.M. Tornero}, `Hironaka's characteristic polygon and effective resolution of surfaces'. {\em Comptes Rendus Mathématiques} 344 (2007) 309--312.


\bibitem{W}
{\sc R.J. Walker}, `Reduction of the Singularities of an algebraic surface'. {\em Ann. of Math.} 36 (1935) 336--365.

\bibitem{Z1}
{\sc O. Zariski}, `The reduction of the singularities of an algebraic surface'. {\em Ann. of Math.} 40 (1936) 639--689.





\bibitem{Z2}
{\sc O. Zariski}, `Reduction of singularities of algebraic three dimensional varieties' {\em Ann. of Math.} 45 (1944) 472--542.
\end{thebibliography}
\end{document}